\documentclass[a4paper,12pt]{amsart}
\usepackage[latin1]{inputenc}  
\usepackage[T1]{fontenc} 
\usepackage{latexsym}
\usepackage{amsmath}
\usepackage{amscd}
\usepackage[all]{xy}
\usepackage{graphics}
\usepackage{amssymb,amsthm,amsfonts}
\usepackage{syntonly}
\usepackage{pdflscape}

\theoremstyle{plain}
\newcommand{\sO}{{\mathcal O}}

\newcommand{\sV}{{\mathcal V}}

\newtheorem{theorem}{Theorem}[section]
\newtheorem{prop}[theorem]{Proposition}
\newtheorem{cor}[theorem]{Corollary}
\newtheorem{lemma}[theorem]{Lemma}

\theoremstyle{definition}

\newtheorem{defi}[theorem]{Definition}

\newtheorem{ex}[theorem]{Example}

\newcommand{\B}{{\mathbb B}}

\newcommand{\C}{{\mathbb C}}
\newcommand{\Q}{{\mathbb Q}}
\newcommand{\R}{{\mathbb R}}

\newcommand{\Z}{{\mathbb Z}}

\newcommand{\V}{{\mathbb V}}
\newcommand{\U}{{\mathbb U}}
\newcommand{\W}{{\mathbb W}}

\newcommand{\p}{{\mathbb P}}

\title[Mixed Hodge Complexes and $L^2$--Cohomology]{Mixed Hodge Complexes and $L^2$--Cohomology for local systems on ball quotients}

\author{Stefan M\"uller-Stach, Xuanming Ye, Kang Zuo }
\email{stach@uni-mainz.de, zuok@uni-mainz.de, ye001@uni-mainz.de}
\begin{document}
\begin{abstract}

We study the $L^2$--cohomology of certain local systems on non-compact arithmetic ball quotients $X=\Gamma \backslash \B_n$.
In the case of a ball quotient surface $X$ we show that vanishing theorems for $L^2$--cohomology are intimately related to vanishing theorems of the type
$$
H^0(\overline{X}, S^n \Omega^1_{\overline{X}}(\log D) \otimes{\mathcal O}_{\overline{X}}(-D) \otimes (K_{\overline{X}}+D)^{-m/3})=0
$$
for $m \ge n \ge 1$ on the toroidal compactification $(\overline{X},D)$.
We also give generalizations to higher dimensional ball quotients and study the mixed Hodge structure
on the sheaf cohomology of a local system in general with the $L^2$-cohomology contributing to the lowest weight part.
\end{abstract}


\maketitle

\section*{Introduction}

Miyaoka~\cite{miyaoka} proved the vanishing
$$
H^0(X, S^n \Omega^1_X \otimes L^{-m})=0
$$
for  $m \ge n \ge 1$, if $X=\Gamma \backslash \B_2$ is a compact ball quotient surface
such that $K_X=\det \Omega^1_X=L^{\otimes 3}$ for some nef and big line bundle $L$.

In this paper we want to investigate more generally vanishing and non--vanishing results for the $L^2$--cohomology of
non--compact ball quotients $X=\Gamma \backslash \B_n$ in any dimension. These theorems are related
to questions about the representation theory of $SU(n,1)$.
The vanishing and non-vanishing theorems on the representation theory side were shown by Li-Schwermer \cite{ls} and Saper \cite{saper}.

The complex ball $\B_n$ is a bounded symmetric domain of type $G/K$ with $G=SU(n,1)$ and $K=U(n)$.
A \emph{ball quotient} $X=\Gamma \backslash \B_n$ is a quotient of $\B_n$ by a torsion--free discrete subgroup $\Gamma \subset SU(n,1)$.
If $\Gamma \subset SU(n,1)$ is an arithmetic subgroup, then $X$ is quasi--projective and allows a natural normal projective compactification
by adding a finite number of points (cusps) at infinity, the Baily--Borel--Satake compactification.
Canonical desingularizations $\overline{X}$ of this compactification are given by toroidal compactifications.

Examples of ball quotient surfaces are \emph{Picard modular surfaces}. Those are (components of) Shimura
varieties which parametrize abelian $3$--folds with given endomorphism algebra \cite{pms}.
In this case $\Gamma$ is a subgroup of $SU(2,1)$ with values in integers of an imaginary quadratic field $E$.
In the classical situation, already studied by Picard \cite{picard}, the abelian $3$--folds are
Jacobians of Picard curves $y^3=P(x)$ with $\deg(P)=4$.
In this way Picard modular surfaces arise as one case in Deligne--Mostow's list \cite{dm}.
We refer to \cite{mmwyz} for an explicit geometric example defined over $\Q(\sqrt{-3})$.

The toroidal compactification in the Picard modular surface case is given by resolving all cusps in a unique way
by smooth elliptic curves \cite{pms}.
In general, arithmetic ball quotients $X=\Gamma \backslash \B_n$ may be embedded into ${\mathcal A}_g$, the moduli space of
abelian varieties, for a suitable $g$. In this case the toroidal
compactification $\overline{X}$ is unique with disjoint abelian varieties as boundary strata \cite{he, yo}.

\medskip
In this paper we study $L^2$--cohomologies of complex local systems $\V$ on non--compact arithmetic
ball quotients $X=\Gamma \backslash \B_n$ with a toroidal compactification $\overline{X}$.
We will assume throughout this paper that $\Gamma$ is torsion-free,
so that $X$ is smooth and $\Gamma$ is the fundamental group. There is the standard representation
$$
\rho: \Gamma \longrightarrow GL(n+1,\C)
$$
and we assume that it has unipotent monodromies around each component of the boundary divisor $D$ of a toroidal
compactification $\overline{X}$ of $X$. This can always be achieved by passing to a subgroup of finite index in $\Gamma$.
We denote the associated local system on $X$ by $\V_1$, and  call it {\sl uniformizing}. The dual local system is denoted by $\V_2$.
Both local systems have rank $n+1$ on $X$. The sum of all Galois conjugates of $\V_1$ contains $\V_2$ and is defined over $\Q$.

The uniformizing local system $\V_1$ underlies a complex VHS on $X$, also denoted by $\V_1$, which is defined over a number field.
The vector bundle $\V_1 \otimes {\mathcal O}_X$ has a canonical Deligne extension to a vector bundle
${\mathcal V}_1$ on $\overline{X}$ together with holomorphic subbundles $F^p \subset {\mathcal V}_1$. It carries
a logarithmic connection $\overline{\nabla}: {\mathcal V}_1 \to {\mathcal V}_1 \otimes \Omega^1_{\overline{X}}(\log D)$.

To $\V_1$, or more generally to any local system $\W$ on $X$ underlying a complex VHS, there
corresponds a \emph{logarithmic Higgs bundle} $E$ on $\overline{X}$ via the Simpson correspondence.
$E$ can be defined as the graded object
\[
E=\bigoplus E^{p,q},
\]
where
\[
E^{p,q}=F^p/F^{p+1}.
\]
If $\W=\V_1$ was the uniformizing local system on $X$, $E$ is called the uniformizing Higgs bundle.
The associated Higgs operator $\theta={\rm gr} \overline{\nabla}$ is a homomorphism of vector bundles
$$
\theta: E \longrightarrow E \otimes \Omega_{\overline{X}}^1(\log D),
$$
which satisfies $\theta \wedge \theta=0$, and Griffiths transversality, i.e.,
$\theta(E^{p,q}) \subset E^{p-1,q+1} \otimes \Omega_{\overline{X}}^1(\log D)$.

Any Higgs bundle $E=\bigoplus E^{p,q}$ gives rise to a complex of vector bundles
$$
E {\buildrel \theta \over \longrightarrow} E \otimes \Omega_{\overline{X}}^1(\log D)
{\buildrel \wedge \theta \over \longrightarrow} E \otimes \Omega_{\overline{X}}^2(\log D)
{\buildrel \wedge \theta \over \longrightarrow} \cdots {\buildrel \wedge \theta \over \longrightarrow}
E \otimes \Omega_{\overline{X}}^{\dim(X)}(\log D).
$$
The hypercohomology of this complex computes the sheaf cohomology
$$
H^*(X,\W)
$$
and the complex itself carries a weight filtration (see \cite{El2} for both facts), giving rise to a MHS on $H^*(X,\W)$, defined first by Deligne, Zucker and Saito.
There is an algebraic subcomplex
$$
\Omega_{(2)}^*(E) \subset E \otimes \Omega_{\overline{X}}^*(\log D)
$$
whose hypercohomology is isomorphic to the $L^2$- resp. intersection cohomology of $\W$
$$
H^*_{L^2}(X,\W)=IH^*(\overline{X},\W)
$$
by a result of Zucker \cite{Zuc2} and Jost-Yang-Zuo \cite{jyz}. We denote this cohomology group also by
$$
H^*_{L^2}(\overline{X},(E,\theta)).
$$
There is a natural map from $L^2$--cohomology to sheaf cohomology, factoring surjectively through the weight zero part \cite{ps}:
\[
 H^*_{L^2}(X,\W) \twoheadrightarrow W_0 H^*(X,\W) \hookrightarrow H^*(X,\W).
\]

\medskip
Now let us describe the results of this paper: First we consider the case $n=2$ of arithmetic ball quotient surfaces
$X=\Gamma \backslash \B_2$ where $\Gamma \subset SU(2,1)$ is an arithmetic subgroup.
Over $X$ one has two complex local systems $\V_1$ and $\V_2$ of rank $3$ defined over some number field.
Both are complex conjugate and dual to each other, and the sum $\W=\V_1 \oplus \V_2$ is a real local system.
We will always assume that $\Gamma$ is small enough so that $X$ is smooth and $\V_1$ and $\V_2$ have unipotent monodromies at infinity.
Under this assumption there is a line bundle $L$ on $\overline{X}$ such that $L^3=K_{{\overline X}}+D$ \cite{mmwyz}.
We may assume that $\V_1$ is the uniformizing local system corresponding to the standard representation
of $\Gamma$ in $SU(2,1)$, and we let $(E_1,\theta)$ be the corresponding Higgs bundle.
The Higgs bundle corresponding to $\V_2$ is denoted by $E_2$. Our first result concerns the $L^2$-cohomology of $E_1$:

\begin{theorem} Assume that $\overline{X}$ is an (arithmetic) ball quotient surface with the assumption that $\Gamma$ is sufficiently small. Then
we have:
\begin{itemize}
\item $H^0({\overline X}, \Omega^1_{\overline X}(\log D) \otimes \Omega^1_{\overline X} \otimes L^{-1})=0$ implies $IH^1(X,\V_1)=0$.
\item $IH^1(X,\V_1)=0$ implies $H^0({\overline X}, S^2 \Omega^1_{\overline X} (\log D)\otimes{\mathcal O}_{\overline{X}}(-D) \otimes L^{-1})=0$.
\end{itemize}
If $\Gamma$ is sufficiently small, then $H^1_{L^2}(E_1)=IH^1(X,\V_1)$ can be non--zero.
\end{theorem}

The non--vanishing statement for $\Gamma$ sufficiently small is due to Kazdan \cite{k} and holds for certain types of ball quotients.
However, for compact ball quotients of certain quaternionic types, examples with vanishing cohomology were found by Kottwitz \cite{ko}.
The local systems $\V_1$ and $\V_2$ are Galois conjugate to each other, hence one vanishes if and only if the other does.
However the vanishing of $H^1_{L^2}(E_2)$ has a different interpretation:

\begin{theorem} Assume that $\overline{X}$ is an (arithmetic) ball quotient surface with the assumption that $\Gamma$ is sufficiently small. Then
$IH^1(X,\V_1)=0$ if and only if $IH^1(X,\V_2)=0$. If these vanish, it implies that
$$
H^0({\overline X}, S^2 \Omega^1_{\overline X}(\log D)\otimes{\mathcal O}_{\overline{X}}(-D) \otimes L^{-2})
= H^1({\overline X}, \Omega^1_{\overline X}(\log D)\otimes{\mathcal O}_{\overline{X}}(-D) \otimes L^{-2})=0.
$$
\end{theorem}

In the sequel we also look at different Schur functors applied to $\V_1$, e.g., $S^2 \V^1$
and $\Lambda^3 \V_1$.
More generally, we consider the most general Schur functors $\W_{a,b}$ which are defined as kernels of natural maps
$$
S^a \V_1 \otimes  S^b \V_2 \longrightarrow S^{a-1} \V_1 \otimes S^{b-1} \V_2.
$$
We refer to \cite{fh} for this notation. As a result we get:

\begin{theorem} Assume that $\overline{X}$ is an (arithmetic) ball quotient surface.
Then one has $H^0(\overline{X} ,S^n \Omega_{\overline{X}}^1(\log D)\otimes{\mathcal O}_{\overline{X}}(-D) \otimes L^{-m})=0$ for all $m \ge n \ge 1$.
\end{theorem}

The twist by $(-D)$ in the theorem is too strong and the proof will give a slightly better result.
The proof uses a vanishing theorem of Ragunathan, Li--Schwermer \cite{ls} and Saper \cite{saper}, since
$\W_{a,b}$ has regular highest weight if and only if $a,b >0$.

In section~\ref{highervan} we give generalizations of the previous results to higher--dimensional arithmetic
ball quotients $X=\Gamma \backslash \B_n$.
For example we compute the Higgs cohomology of the uniformizing Higgs bundle $E_1$ and its symmetric powers $S^k E_1$:

\begin{theorem} The $L^2$--Higgs complex for the symmetric powers $(S^kE_1,\theta)$ with $k \ge 1$ is quasi--isomorphic to
$$
0 \to  T^0(k) {\buildrel 0 \over \to}  T^1(k) {\buildrel 0 \over \to} \to \cdots {\buildrel 0 \over \to}  T^n(k) \to 0,
$$
with trivial differentials, where $T^i(k)$ is the sheaf of $L^2$--sections of
$$
{\rm Ker}\left( S^k \Omega_{\overline{X}}^1(\log D) \otimes L^{-k} \otimes
\Omega^i_{\overline{X}}(\log D) \to S^{k-1} \Omega_{\overline{X}}^1(\log D) \otimes L^{-k}
\otimes \Omega^{i+1}_{\overline{X}}(\log D) \right)
$$
for $i=1,\ldots,n$. For $i=0$ we get $T_0(k)=L^{-k}$.\\
\end{theorem}

Here, by sheaves of $L^2$-sections we mean the sections of the subsheaves arising in the subcomplex $\Omega_{(2)}^*(E)$.
We prove a similar theorem for the dual Higgs bundle $E_2$ in the same section.

\medskip
In the final section we look  at the weight filtration $W_\bullet$ on the Higgs complex $E \otimes \Omega^\bullet_{\overline{X}}(\log D)$
on a ball quotient surface $X$ where $E$ is the Higgs bundle corresponding to the real VHS $\W=S^k\V_1 \oplus S^k\V_2$.
Together with the Hodge filtration it computes the mixed Hodge structure on $H^*(X,\W)$. The natural map
$IH^*(\overline{X},\W) \longrightarrow H^*(X,\W)$
maps onto the weight zero part $W_0H^*(X,\W)$ by \cite{ps}. Our goal is to compute the mixed Hodge numbers $h^{p,q}$ of this Hodge structure.
To do this, we first compute the weight graded pieces of the complex $E \otimes \Omega^\bullet_{\overline{X}}(\log D)$. Using this and the comparison map
from intersection cohomology we prove the following results:

\begin{theorem}
The mixed Hodge numbers of $H^{l}(X, S^{k}\V_{1}\oplus
S^{k}\V_{2}), 0\leq l \leq 4$ above weight $l+k$ depend only on
the number of boundary components $h$. More precisely: \\
1. For $l=3$, the only non--zero mixed Hodge number is
$h^{k+2,k+2}_{3}=2h.$ \\
2. For $l=2$, the only non--zero mixed Hodge numbers are
$h^{k+2,1}_{2}=h^{1,k+2}_{2}\leq h$ and
$h^{k+2,k+1}_{2}=h^{k+1,k+2}_{2}= h$. \\
Furthermore, if
$$
H^{1}(X, S^{k}\V_{1}\oplus S^{k}\V_{2})=0,
$$
then $h^{k+2,1}_{2}=h^{1,k+2}_{2} = h.$\\
3. For $l=1$ one has
$$
H^{1}(X, S^{k}\V_{1}\oplus S^{k}\V_{2}) = W_{0}H^{1}(X, S^{k}\V_{1}\oplus S^{k}\V_{2}).
$$
\end{theorem}

\begin{theorem}
The mixed Hodge numbers of $H^{l}(X, {\rm End}^0(\V_1)), 0\leq l
\leq 4$ above weight $l+2$ depend only on the
number of boundary components $h$. More precisely: \\
1. For $l=3$, the nonzero mixed Hodge number is only
$h^{4,4}_{3}=h.$ \\
2. For $l=2$, the nonzero mixed Hodge numbers are only
$h^{4,2}_{2}=h^{2,4}_{2}= h.$ \\
\end{theorem}

\section{The $L^2$--cohomology of certain local systems on Picard modular surfaces}
\label{L2}

In this section we study the $L^2$--cohomology of certain local systems on an arithmetic ball quotient surface, and study their (non)--vanishing.

\subsection{The Simpson correspondence in the simplest case}
\label{higgsbasics}

Let $X$ be a smooth, quasiprojective variety and $\overline{X}$ be a smooth compactification with normal crossing boundary divisor $D$.
There is a categorical correspondence between direct sums of local systems on $X$ and polystable logarithmic Higgs bundles on $\overline{X}$ with vanishing
Chern classes that goes back to Hitchin, Donaldson, Uhlenbeck-Yau and Simpson, see \cite{s} for a reference.
It associates to any irreducible local system $\W$ on $X$ underlying a complex VHS and with unipotent monodromies at infinity a logarithmic Higgs bundle
$(E,\theta)$ where $E$ is a vector bundle on $\overline{X}$ and
\[
 \theta: E \to E \otimes \Omega_{\overline{X}}^1(\log D)
\]
which, in addition, satisfies $\theta \wedge \theta=0$.

Under these assumptions it is constructed as follows: the vector bundle ${\mathcal W}=\W \otimes {\mathcal O}_X$ has a canonical Deligne extension to a vector bundle
$\overline{\mathcal W}$ on $\overline{X}$ together with holomorphic subbundles $F^p \subset \overline{\mathcal W}$.
The associated Higgs bundle $E$ can be defined as the graded object
\[
E=\bigoplus E^{p,q},
\]
where
\[
E^{p,q}=F^p/F^{p+1}
\]
together with the graded extended Gau{\ss}--Manin connection $\overline{\nabla}$ as Higgs operator $\theta$.
The Higgs operator $\theta$ is a homomorphism of vector bundles
$$
\theta: E \longrightarrow E \otimes \Omega_{\overline{X}}^1(\log D),
$$
which satisfies $\theta \wedge \theta=0$ and Griffiths transversality, i.e.,
$\theta(E^{p,q}) \subset E^{p-1,q+1} \otimes \Omega_{\overline{X}}^1(\log D)$.

\subsection{The uniformizing Higgs bundles on Picard modular surfaces}
\label{higgspicard}

Let $\overline{X}$ be a toroidal compactification of an arithmetic ball quotient surface $X=\Gamma \backslash \B_2$ \cite{pms}.
There are two local systems $\V_1$ and $\V_2$ on $X=\overline{X} \setminus D$.
Both are $3$--dimensional, are complex conjugate and dual to each other, and underly a complex VHS which is defined over a number field.
The standard representation of $\Gamma \subset SU(2,1)$ corresponds to $\V_1$ without loss of generality.
We assume throughout that $\Gamma$ is torsion-free and the monodromy at infinity, i.e., around the elliptic cusp divisors, is unipotent.
This implies that $K_{\overline{X}}+D$ is divisible by $3$ and we define $L$ to be a third root, so that
$K_{\overline{X}}+D=L^{\otimes 3}$ \cite{mmwyz}. $L$ is a nef and big line bundle.
The Deligne extensions of the vector bundles ${\mathcal V_i}:=\V_i \otimes {\mathcal O}_{X}$ to $\overline{X}$ are denoted by
$\overline{\mathcal V_i}$. The corresponding Higgs bundle are denoted by $E_i$. The local system $\W=\V_1 \oplus \V_2$ is real and corresponds
to the Higgs bundle $E=E_1 \oplus E_2$. By \cite{mmwyz}
we may assume that the Higgs bundle $E_1$ corresponding to $\V_1$ is
$$
E_1 = E_1^{1,0} \oplus E_1^{0,1} = \left( \Omega^1_{\overline{X}}(\log D) \otimes L^{-1} \right) \oplus L^{-1},
$$
where $L^3=\det \Omega^1_{\overline{X}}(\log D) = {\mathcal O}_{\overline{X}}(K_{\overline{X}}+D)$ and the Higgs bundle corresponding to $\V_2$ is
$$
E_2 = E_2^{1,0} \oplus E_2^{0,1} = L \oplus \left( \Omega^1_{\overline{X}}(\log D) \otimes L^{-2} \right).
$$
Note that $E_2^{0,1}=T_{\overline{X}}(-\log D) \otimes L$. The non--zero part of the Higgs operator for $E_1$ is given by
$$
\theta=id: E_1^{1,0}= \Omega^1_{\overline{X}}(\log D) \otimes L^{-1} \to  E_1^{0,1} \otimes \Omega^1_{\overline{X}}(\log D)
=L^{-1} \otimes \Omega^1_{\overline{X}}(\log D).
$$
For $E_2$
$$
\theta: E_2^{1,0}=L \to E_2^{0,1} \otimes \Omega^1_{\overline{X}}(\log D)=T_{\overline{X}}(-\log D) \otimes L \otimes \Omega^1_{\overline{X}}(\log D)
$$
is the inclusion onto $L \subset T_{\overline{X}}(-\log D) \otimes L \otimes \Omega^1_{\overline{X}}(\log D)$ dual to the contraction operator.
The decomposition of $E$ into Hodge types is therefore
$$
E^{1,0}=E_1^{1,0} \oplus E_2^{1,0}=\left( \Omega^1_{\overline{X}}(\log D) \otimes L^{-1} \right) \oplus L,
$$
$$
E^{0,1}=E_1^{0,1} \oplus E_2^{0,1}=L^{-1} \oplus \left(\Omega^1_{\overline{X}}(\log D) \otimes L^{-2} \right).
$$
Taking the determinant implies:
\begin{lemma} $\det E^{1,0}=L^2$ and
$\det E^{0,1}=L^{-2}$.
\end{lemma}

\medskip For a general arithmetic ball quotient surface $\overline{X}$ in Deligne--Mostow's list \cite{dm} the canonial Higgs bundle associated
to a universal family $f$ can have more direct summands than $\V_1$  and $\V_2$ as in the following example.

\begin{ex}
Consider the family of cyclic $5:1$ covers $C \to \p^1$ of genus $6$ ramified along $5$ points. Then the eigenspace decomposition for
$\Z/5\Z$ implies $R^1f_*\C =\V_1 \oplus \V_2 \oplus \U_1 \oplus \U_2$, where $\U_1, \U_2$ are unitary local systems. We may again assume that $\V_1$
is the standard representation. All four local systems have rank three and are Galois conjugate to each other.
The unitary local systems correspond to Higgs bundles with $\theta=0$.
\end{ex}

\subsection{Definition of $L^2$--Higgs cohomology on surfaces} \ \\

For ball quotients $\overline{X}$ of arbitrary dimension we have that the
boundary divisor $D$ in the toroidal compactification is smooth~\cite{he,yo}.

\medskip
The $L^2$--Higgs cohomology $H^i_{L^2}(\overline{X}, (E,\theta))$ of any logarithmic Higgs bundle $(E,\theta)$
on a surface $\overline{X}$ may be computed via the hypercohomology of a complex of algebraic sheaves \cite{jyz, Zuc2}
$$
\Omega^0(E)_{(2)} \to  \Omega^1(E)_{(2)} \to  \Omega^2(E)_{(2)},
$$
In our case, where $D$ is smooth in particular, one has \cite[Appendix]{mmwyz}:
\begin{eqnarray*}
\Omega^0(E)_{(2)} &=&{\text{Ker}}N_1={\text{Ker}}({\text{Res}}(\theta)), \text{ where } N_1={\text{Res}}(\theta), \\
\Omega^1(E)_{(2)} &=&{\frac{dz_1}{z_1}}\otimes z_1 E+dz_2\otimes {\text{Ker}}N_1=dz_1 \otimes E + dz_2 \otimes {\text{Ker}}N_1,\\
\Omega^2(E)_{(2)} &=&{\frac{dz_1}{z_1}}\wedge dz_2\otimes z_1 E= \Omega^2_{\overline{X}}\otimes E.
\end{eqnarray*}

The shape of these sheaves arises from $L^2$--conditions in the Poincar\'e metric
on $X$ \cite{jyz, Zuc2}. We have an isomorphism of cohomology groups
$$
H^i_{L^2}(\overline{X}, (E,\theta)) = IH^i(\overline{X},\W)
$$
for any Higgs bundle $(E,\theta)$ underlying a complex VHS with local system $\W$ \cite{jyz,Zuc2}, i.e.,
$L^2$--Higgs and intersection cohomology on $\overline{X}$ are isomorphic.

\medskip

Now let $E_1=\Omega^1_{\overline X}(\log D)\otimes L^{-1} \oplus L^{-1}$
be the uniformizing Higgs bundle as above. Taking $v$ as the generating section of $L^{-1}$,
${\frac{dz_1}{z_1}}\otimes v, dz_2\otimes v$ as the generating
sections of $\Omega^1_{\overline X}(\log D)\otimes L^{-1}$, then the Higgs field
\[
\theta: E\to E \otimes\Omega^1_{\overline X}(\log D)
\]
is defined by setting $\theta({\frac{dz_1}{z_1}}\otimes
v)=v\otimes{\frac{dz_1}{z_1}}$, $\theta(dz_2\otimes v)=v\otimes
dz_2$, and $\theta(v)=0$. If $\theta$ is written as
$N_1{\frac{dz_1}{z_1}}+N_2dz_2$, then
$N_1({\frac{dz_1}{z_1}}\otimes v)=v$, $N_1(dz_2\otimes v)=0$,
$N_1(v)=0$, $N_2({\frac{dz_1}{z_1}}\otimes v)=0$, $N_2(dz_2\otimes
v)=v$, $N_2(v)=0$; the kernel of $N_1$ is the subsheaf generated
by $dz_2\otimes v$, $\frac {dz_1}{z_1} \otimes z_1 v$  and $v$, hence
$\left( \Omega^1_{\overline X}\otimes L^{-1} \right) \oplus L^{-1}$.

\medskip
Summarizing, we have \cite[Appendix]{mmwyz}
$$
\Omega^0(E)_{(2)} ={\text{Ker}}N_1=\left( \Omega^1_{\overline X}\otimes L^{-1} \right) \oplus L^{-1}, \quad
\Omega^2(E)_{(2)} = \Omega^2_{\overline{X}}\otimes E,
$$
and
$$
E \otimes \Omega^1_{\overline X}(\log D)(-D) \subseteq \Omega^1(E)_{(2)} \subseteq E \otimes \Omega^1_{\overline X}.
$$

\subsection{The $L^2$--Higgs cohomology of $E_1$} \ \\

In this subsection we compute the $L^2$--cohomology of the uniformizing Higgs bundle $E_1$ on an arithmetic
ball quotient surface $\overline {X}$: First neglecting $L^2$--conditions, the complex
$$
(E_1^\bullet,\theta): E_1 {\buildrel \theta \over \to} E_1 \otimes \Omega^1_{\overline X}(\log D)
{\buildrel \theta \over \to} E_1 \otimes \Omega^2_{\overline X}(\log D)
$$
looks like:
\begin{tiny}
$$
\begin{matrix}
&&&& \left(\Omega^1_{\overline X}(\log D) \otimes L^{-1}\right) & \oplus & L^{-1} \cr
&&&& \downarrow \cong  && \downarrow   \cr
&&\left(\Omega^1_{\overline X}(\log D)^{\otimes 2} \otimes L^{-1} \right)
& \oplus & \left(L^{-1} \otimes \Omega^1_{\overline X}(\log D) \right) & & 0  \cr
&&\downarrow   &&&&  \cr
\left(\Omega^1_{\overline X}(\log D) \otimes L^{-1} \otimes \Omega^2_{\overline X}(\log D) \right) &  \oplus &
\left( L^{-1} \otimes \Omega^2_{\overline X}(\log D)\right). & & & &
\end{matrix}
$$
\end{tiny}
Therefore it is quasi--isomorphic to a complex
$$
L^{-1} {\buildrel 0 \over \longrightarrow} S^2 \Omega^1_{\overline X}(\log D) \otimes L^{-1} {\buildrel 0 \over \longrightarrow}
\Omega^1_{\overline X}(\log D) \otimes \Omega^2_{\overline X}(\log D) \otimes L^{-1}
$$
with trivial differentials. As $L$ is nef and big, we have
$$
H^0(L^{-1})=H^1(L^{-1})=0.
$$
Hence we get
$$
{\mathbb H}^1({\overline X},(E_1^\bullet,\theta))
\cong H^0({\overline X}, S^2 \Omega^1_{\overline X}(\log D) \otimes L^{-1})
$$
and ${\mathbb H}^2({\overline X},(E_1^\bullet,\theta))$ is equal to
\begin{small}
$$
H^0({\overline X},K_{\overline X} \otimes L)^\vee \oplus
H^0({\overline X},\Omega^1_{\overline X}(\log D) \otimes L^{2})
\oplus H^1({\overline X}, S^2 \Omega^1_{\overline X}(\log D) \otimes L^{-1}).
$$
\end{small}
If we now impose the $L^2$--conditions and use the complex
$\Omega^*_{(2)}(E_1)$ instead of $(E_1^\bullet,\theta)$, the resulting cohomology groups are subquotients
of the groups described above.

\begin{theorem} \label{vanishingtheorem} With the assumption on $X$ as above:
\begin{itemize}
\item $H^0({\overline X}, \Omega^1_{\overline X}(\log D) \otimes \Omega^1_{\overline X} \otimes L^{-1})=0$ implies $IH^1(X,\V_1)=0$.
\item $IH^1(X,\V_1)=0$ implies $H^0({\overline X}, S^2 \Omega^1_{\overline X}(\log D)\otimes{\mathcal O}_{\overline{X}}(-D) \otimes L^{-1})=0$.
\end{itemize}
If $\Gamma$ is sufficiently small, then $H^1_{L^2}(E_1)=IH^1(X,\V_1)$ can be non--zero.
\end{theorem}

\proof By the proof above,
$$
H^1_{L^2}({\overline X},E_1)=L^2-\text{sections in } H^0({\overline X}, S^2 \Omega^1_{\overline X}(\log D) \otimes L^{-1}).
$$
In addition, one certainly has $\Omega^1_{\overline X}(\log D)(-D) \otimes E_1 \subseteq \Omega^1(E_1)_{(2)}  \subseteq \Omega^1_{\overline X} \otimes E_1$.
Together with $E_1=\Omega^1_{\overline X}(\log D) \otimes L^{-1} \oplus L^{-1}$ this implies that
\begin{small}
$$
H^0({\overline X}, S^2 \Omega^1_{\overline X}(\log D)\otimes{\mathcal O}_{\overline{X}}(-D) \otimes L^{-1}) \subset H^1_{L^2}({\overline X},E_1)
\subset  H^0({\overline X}, \Omega^1_{\overline X}(\log D) \otimes \Omega^1_{\overline X} \otimes L^{-1})
$$
\end{small}
which gives the assertion.
The standard representation $\V_1$ can have non--vanishing cohomology if the arithmetic subgroup $\Gamma$ has large index in $SU(2,1)$.
This is a result of Kazdan, see \cite[Cor. 5.9 and Remark 5.10]{bw}.
\endproof

In \cite{mmwyz} such a vanishing result has been verified for one particular example of a Picard modular surface,
studied also by Hirzebruch and Holzapfel.

\subsection{The $L^2$--Higgs cohomology of $E_2$} \ \\

We compute the $L^2$--cohomology of $E_2$ in a similar way: First neglecting
$L^2$--conditions, the complex
$$
(E_2^\bullet,\theta): E_2 {\buildrel \theta \over \to} E_2 \otimes \Omega^1_{\overline X}(\log D)
{\buildrel \theta \over \to} E_2 \otimes \Omega^2_{\overline X}(\log D)
$$
looks like:
$$
\begin{matrix}
&&& L & \oplus \left(\Omega^1_{\overline X}(\log D) \otimes L^{-2}\right) \cr
&&& \downarrow  & \downarrow &  \cr
&\left(L \otimes \Omega^1_{\overline X}(\log D) \right) & \oplus & \left( \Omega^1_{\overline X}(\log D)^{\otimes 2} \otimes L^{-2}  \right) & 0 &  \cr
&\downarrow  \cong &&&&  \cr
L^4 & \oplus \left(\Omega^1_{\overline X}(\log D) \otimes L \right) & & & &.
\end{matrix}
$$
and the arrow $L \to \Omega^1_{\overline X}(\log D)^{\otimes 2} \otimes L^{-2}$ is injective with cokernel
$S^2 \Omega^1_{\overline X}(\log D) \otimes L^{-2}$.
Therefore $(E_2^\bullet,\theta)$ is quasi--isomorphic to a complex
$$
\Omega^1_{\overline X}(\log D) \otimes L^{-2} {\buildrel 0 \over \longrightarrow}
S^2 \Omega^1_{\overline X}(\log D) \otimes L^{-2} {\buildrel 0 \over \longrightarrow} L^{4}
$$
with trivial differentials. Hence we get
$$
{\mathbb H}^1({\overline X},(E_2^\bullet,\theta))
\cong H^0({\overline X}, S^2 \Omega^1_{\overline X}(\log D) \otimes L^{-2}) \oplus
H^1({\overline X}, \Omega^1_{\overline X}(\log D) \otimes L^{-2})
$$
and ${\mathbb H}^2({\overline X},(E_1^\bullet,\theta))$ is equal to
$$
H^0({\overline X}, L^4) \oplus H^1({\overline X}, S^2 \Omega^1_{\overline X}(\log D) \otimes L^{-2}) \oplus
H^2({\overline X},\Omega^1_{\overline X}(\log D) \otimes L^{2}).
$$
With $L^2$--conditions we have to again introduce twists by $-D$ as above.
The vanishing of $H^1_{L^2}(E_1)$ and $H^1_{L^2}(E_2)$ is related by
a conjugation argument:
\begin{theorem} Assume that $\overline{X}$ is an arithmetic ball quotient surface. Then
$H^1_{L^2}(E_1)=0$ if and only if $H^1_{L^2}(E_2)=0$. If both vanish, this implies that
$$
H^0({\overline X}, S^2 \Omega^1_{\overline X}(\log D)(-D) \otimes L^{-2})
= H^1({\overline X}, \Omega^1_{\overline X}(\log D)(-D) \otimes L^{-2})=0.
$$
\end{theorem}

\begin{proof} The first statement follows from complex conjugation, the second one was shown above.
\end{proof}

In \cite{mmwyz} there is an example with $H^1_{L^2}(E_1)=H^1_{L^2}(E_2)=0$.
In general this group does not vanish by the result of Kazdan \cite[Cor. 5.9 and Remark 5.10]{bw}.

\subsection{The $L^2$--Higgs cohomology of ${\rm End}^0(\V_1)={\rm End}^0(\V_2)$} \ \\

If we consider irreducible representations other than $S^n \V_1$ or $S^n \V_2$ then there is the following vanishing theorem
which is a special case of the results in \cite{ls,saper}. Recall that we assumed $\Gamma$ to be torsion-free.

\begin{theorem}[Ragunathan, Li--Schwermer, Saper]
Let $\W$ be an irreducible representation of an arithmetic subgroup $\Gamma \subset SU(2,1)$ coming from $G=SU(2,1)$,
i.e., a local system on $X$. If the highest weight of $\W$ is regular, then one has $IH^1(\overline{X},\W)=0$.
\end{theorem}

For representations of $SU(2,1)$ regular highest weight is equivalent to not being isomorphic to either $S^n \V_1$ nor $S^n \V_2$.
This applies in particular to all representations $\W_{a,b}$ of $SU(2,1) \subseteq SL_3(\C)$ with $a,b>0$. These are defined as
kernels of natural maps
$$
S^a \V_1 \otimes  S^b \V_2 \longrightarrow S^{a-1} \V_1 \otimes S^{b-1} \V_2.
$$
We refer to \cite{fh} for this notation. A consequence of the vanishing theorem is:

\begin{cor}\label{vancor}
$IH^1(\overline{X},\W_{a,b})=0$ for $a,b >0$.
\end{cor}

The Higgs bundle corresponding to ${\rm End}(\V_1)$ is $E_1 \otimes E_2$. It is of weight
$2$ with Hodge types $(2,0)+(1,1)+(0,2)$. We have
$$
E_1 \otimes E_2 = \left( \Omega^1_{\overline X}(\log D) \otimes L^{-1} \oplus L^{-1} \right)
\otimes \left(L \oplus  \Omega^1_{\overline X}(\log D) \otimes L^{-2} \right)
$$
$$
=\Omega^1_{\overline X}(\log D) \oplus \left( \Omega^1_{\overline X}(\log D)^{\otimes 2} \otimes L^{-3} \oplus {\mathcal O}_{\overline{X}} \right)
\oplus \left(\Omega^1_{\overline X}(\log D) \otimes L^{-3} \right)
$$
corresponding to types. The Higgs bundle corresponding to $\W_{1,1}={\rm End}^0 \V_1$ is the quotient of this bundle modulo
the image of ${\mathcal O}_{\overline{X}}$ and hence is isomorphic to
$$
\Omega^1_{\overline X}(\log D) \oplus \left( \Omega^1_{\overline X}(\log D) \otimes \Omega^1_{\overline X}(\log D) \otimes L^{-3} \right)
\oplus \left(\Omega^1_{\overline X}(\log D) \otimes L^{-3} \right).
$$
Its Higgs complex looks like
$$
\begin{matrix}
\Omega^1_{\overline X}(\log D) \oplus \left( \Omega^1_{\overline X}(\log D) \otimes \Omega^1_{\overline X}(\log D) \otimes L^{-3} \right)
\oplus \left(\Omega^1_{\overline X}(\log D) \otimes L^{-3} \right)
\cr \downarrow  \cr
\Omega^1_{\overline X}(\log D)^{\otimes 2} \oplus
\left( \Omega^1_{\overline X}(\log D)  \otimes \Omega^1_{\overline X}(\log D) \otimes \Omega^1_{\overline X}(\log D) \otimes L^{-3}\right)
\oplus \left(\Omega^1_{\overline X}(\log D)^{\otimes 2} \otimes L^{-3} \right)
\cr \downarrow  \cr
\left( \Omega^1_{\overline X}(\log D)\otimes L^3  \right) \oplus  \left(\Omega^1_{\overline X}(\log D)\otimes \Omega^1_{\overline X}(\log D)  \right)
\oplus \left(\Omega^1_{\overline X}(\log D) \right)
\end{matrix}
$$
This complex is quasi--isomorphic to
$$
\Omega^1_{\overline X}(\log D) \otimes L^{-3}  {\buildrel 0 \over \longrightarrow}
S^3  \Omega^1_{\overline X}(\log D) \otimes L^{-3} {\buildrel 0 \over \longrightarrow}
\Omega^1_{\overline X}(\log D)\otimes L^3.
$$
Corollary~\ref{vancor} with $a=b=1$ now implies a rigidity theorem:

\begin{theorem}[Weil rigidity] Let $X$ be an arithmetic ball quotient surface with uniformizing VHS $\V_1$.  Then
$IH^1(\overline{X},{\rm End}^0 \V_1)=0$ and therefore
$$
H^0(\overline{X}, S^3  \Omega^1_{\overline X}(\log D)(-D) \otimes L^{-3})=0
$$
as well as
$$
H^1(\overline{X}, \Omega^1_{\overline X}(\log D) \otimes L^{-3})=H^1(\overline{X}, T_{\overline{X}}(-\log D))=0.
$$
\end{theorem}

\subsection{The $L^2$--Higgs complex of $S^2 E_1$} \ \\

Let us look at the symmetric square $S^2 \V_1$. The associated Higgs bundle is $S^2E_1$.
The Higgs complex without $L^2$--conditions looks as follows:
\begin{small}
$$
\begin{matrix}
\left(S^2 \Omega^1_{\overline X}(\log D) \otimes L^{-2} \right) \oplus \left( \Omega^1_{\overline X}(\log D) \otimes L^{-2} \right)
\oplus L^{-2} \cr \downarrow  \cr
\left( S^2 \Omega^1_{\overline X}(\log D) \otimes L^{-2} \otimes \Omega^1_{\overline X}(\log D) \right)
\oplus \left( \Omega^1_{\overline X}(\log D) \otimes L^{-2} \otimes \Omega^1_{\overline X}(\log D) \right)  \oplus
\left(L^{-2}  \otimes \Omega^1_{\overline X}(\log D) \right) \cr \downarrow  \cr
\left( S^2 \Omega^1_{\overline X}(\log D) \otimes L^{-2} \otimes \Omega^2_{\overline X}(\log D) \right)
\oplus \left( \Omega^1_{\overline X}(\log D) \otimes L^{-2} \otimes \Omega^2_{\overline X}(\log D) \right)  \oplus
\left(L^{-2}  \otimes \Omega^2_{\overline X}(\log D) \right)
\end{matrix}
$$
\end{small}
Again, many differentials in this complex are isomorphisms or zero. For example the differential
$$
S^2 \Omega^1_{\overline X}(\log D) \otimes L^{-2} \otimes \Omega^1_{\overline X}(\log D)
\to \Omega^1_{\overline X}(\log D) \otimes L^{-2} \otimes \Omega^2_{\overline X}(\log D)
$$
is a projection map onto a direct summand, since for every vector space $W$ we have the
identity
$$
S^2 W \otimes W = S^3 W \oplus \left( W \otimes \Lambda^2 W \right).
$$
Therefore the Higgs complex for $S^2(E_1)$ is quasi--isomorphic to
$$
L^{-2} {\buildrel 0 \over \to}  S^3 \Omega^1_{\overline X}(\log D) \otimes L^{-2}
{\buildrel 0 \over \to} S^2 \Omega^1_{\overline X}(\log D) \otimes L.
$$
We conclude that the first cohomology is given by a subspace
$$
H^1_{L^2}(S^2 E_1) \subseteq H^0({\overline X},S^3 \Omega^1_{\overline X}(\log D) \otimes L^{-2}).
$$

\subsection{The $L^2$--Higgs complex of $S^2 E_2$} \ \\

Now look at the symmetric square $S^2 \V_2$. The associated Higgs bundle is $S^2E_2$.
The Higgs complex ignoring $L^2$--conditions looks as follows:
$$
\begin{matrix}
L^2 \oplus \left( \Omega^1_{\overline X}(\log D) \otimes L^{-1} \right) \oplus \left( S^2 \Omega^1_{\overline X}(\log D) \otimes L^{-4} \right)
\cr \downarrow  \cr
\left( L^{\otimes 2} \otimes  \Omega^1_{\overline X}(\log D) \right) \oplus \left( \Omega^1_{\overline X}(\log D)^{\otimes 2} \otimes L^{-1} \right)
\oplus \left( S^2 \Omega^1_{\overline X}(\log D) \otimes L^{-4} \otimes \Omega^1_{\overline X}(\log D) \right)
\cr \downarrow  \cr
L^5 \oplus \left( \Omega^1_{\overline X}(\log D) \otimes L^2 \right) \oplus
\left( S^2 \Omega^1_{\overline X}(\log D) \otimes L^{-1} \right)
\end{matrix}
$$
Again, many differentials in this complex are isomorphisms, exact  or zero.
For example the sequence
$$
L^2 \to \Omega^1_{\overline X}(\log D)^{\otimes 2} \otimes L^{-1} \to S^2 \Omega^1_{\overline X}(\log D) \otimes L^{-1}
$$
is exact and $\Omega^1_{\overline X}(\log D) \otimes L^2 $ is mapped isomorphically. By the plethysm
$$
S^2 W \otimes W = S^3 W \oplus \left( W \otimes \Lambda^2 W \right).
$$
we get that the Higgs complex is quasi--isomorphic to
$$
S^2 \Omega^1_{\overline X}(\log D) \otimes L^{-4} {\buildrel 0 \over \to} S^3 \Omega^1_{\overline X}(\log D) \otimes L^{-4}
{\buildrel 0 \over \to} L^5.
$$

We will later see that in this complex one has
$$
H^0({\overline X},S^3 \Omega^1_{\overline X}(\log D)(-D) \otimes L^{-4})=0.
$$
The proof will be given in section~\ref{vansection}.

\subsection{The $L^2$--Higgs complex for $\Lambda^3 \V$} \ \\

The Higgs bundle corresponding to the third primitive cohomology inside
$\Lambda^3 \V$ will be denoted by $E^3_{\rm pr}=\bigoplus E_{\rm
pr}^{p,q}$.
We have in particular two important graded pieces: \\

(A) $E_{\rm pr}^{2,1} \to  E_{\rm pr}^{1,2} \otimes
\Omega_{\overline{X}}^1(\log D) \to E_{\rm pr}^{0,3} \otimes L^3$, and \\
(B) $E_{\rm pr}^{3,0} \to  E_{\rm pr}^{2,1} \otimes
\Omega_{\overline{X}}^1(\log D) \to E_{\rm pr}^{1,2} \otimes L^3$. \\
\ \\
Let us first compute all $E_{\rm pr}^{p,q}$. We have $E_{\rm
pr}^{3,0}=L^2$ and $E_{\rm pr}^{0,3}=L^{-2}$.
Furthermore
$$
E^{2,1}= \Lambda^2 E^{1,0} \otimes E^{0,1}=
{\mathcal O}_{\overline{X}} \oplus 2 \left( \Omega_{\overline{X}}^1(\log D) \otimes L^{-1}\right) \oplus
\left( \Omega_{\overline{X}}^1(\log D)^{\otimes 2} \otimes L^{-2} \right).
$$
Hence we get the primitive part
$$
E_{\rm pr}^{2,1} = {\mathcal O}_{\overline{X}} \oplus \left(
\Omega_{\overline{X}}^1(\log D) \otimes L^{-1}\right) \oplus \left(S^2 \Omega_{\overline{X}}^1(\log D) \otimes L^{-2}\right),
$$
since the Lefschetz operator $\wedge \omega$ is a natural inclusion here. In a similar way we get
$$
E_{\rm pr}^{1,2}= {\mathcal O}_{\overline{X}} \oplus
\left(\Omega_{\overline{X}}^1(\log D) \otimes L^{-2}\right) \oplus \left(S^2 \Omega_{\overline{X}}^1(\log D) \otimes L^{-4}\right).
$$
Now the complex (A) becomes
$$
\begin{matrix}
{\mathcal O}_{\overline{X}} \oplus \left( \Omega_{\overline X}^1(\log D) \otimes L^{-1}\right)  \oplus \left( S^2 \Omega_{\overline{X}}^1(\log D) \otimes L^{-2}\right)
\cr \downarrow  \cr
\Omega^1_{\overline X}(\log D) \oplus \left( \Omega^1_{\overline X}(\log D)^{\otimes 2} \otimes L^{-2} \right) \oplus
\left( S^2 \Omega_{\overline{X}}^1(\log D) \otimes \Omega^1_{\overline X}(\log D) \otimes L^{-4} \right)
\cr \downarrow  \cr
L.
\end{matrix}
$$
It is quasi--isomorphic to
$$
(A): \quad {\mathcal O} _{\overline X} {\buildrel 0 \over \to}
\left( S^3 \Omega^1_{\overline X}(\log D) \otimes L^{-4}\right) \oplus \Omega_{\overline X}^1(\log D)  \to 0
$$
in degrees $0$ and $1$ only. In a similar way (B) becomes
$$
\begin{matrix}
L^2 \cr \downarrow  \cr
\Omega^1_{\overline X}(\log D) \oplus \left( \Omega^1_{\overline X}(\log D)^{\otimes 2} \otimes L^{-1} \right) \oplus
\left( S^2 \Omega_{\overline{X}}^1(\log D) \otimes \Omega^1_{\overline X}(\log D) \otimes L^{-2} \right)
\cr \downarrow  \cr
L^3 \oplus \left( \Omega_{\overline X}^1(\log D) \otimes L^{1}\right)  \oplus \left( S^2 \Omega_{\overline X}^1(\log D) \otimes L^{-1}\right)
\end{matrix}
$$
It is quasi--isomorphic to
$$
(B): \quad  \left( S^3 \Omega^1_{\overline X}(\log D) \otimes L^{-2} \right) \oplus \Omega_{\overline X}^1(\log D)  {\buildrel 0 \over \to} L^3
$$
in degrees $1$ and $2$ only.

\subsection{A vanishing theorem} \label{vansection} \ \\

 For higher symmetric powers on arithmetic
ball quotient surfaces we get the following vanishing theorem:

\begin{theorem}\label{vantheorem}
One has $H^0(\overline{X} ,S^n \Omega_{\overline{X}}^1(\log D)(-D) \otimes L^{-m})=0$ for all $m \ge n \ge 3$.
\end{theorem}
\proof Consider $\W_{a,b}$ for $a,b>0$. The corresponding Higgs bundle $E_{a,b}$ is a subbundle of $S^a E_1 \otimes S^b E_2$.
Since $E_1= \left( \Omega^1_{\overline{X}}(\log D) \otimes L^{-1} \right) \oplus L^{-1}$ and
$E_2= L \oplus \left( \Omega^1_{\overline{X}}(\log D) \otimes L^{-2} \right)$, $S^a E_1 \otimes S^b E_2$ contains the direct summand
$S^a \Omega^1_{\overline{X}}(\log D) \otimes  L^{-a} \otimes S^b \Omega^1_{\overline{X}}(\log D) \otimes L^{-2b}$.
We have defined
$$
E_{a,b}:={\rm Ker} (S^a E_1 \otimes S^b E_2 {\buildrel \iota \over \to} S^{a-1} E_1 \otimes S^{b-1} E_2),
$$
with the map $\iota$ induced by the pairing $E_1 \otimes E_2 \to {\mathcal O}_{\overline X}$.
When one restricts the pairing $E_1 \otimes E_2 \to {\mathcal O}_{\overline X}$ to $E_1^{1,0} \otimes E_2^{0,1}$
it is given by the wedge product
$$
\Omega^1_{\overline{X}}(\log D) \otimes L^{-1} \otimes \Omega^1_{\overline{X}}(\log D) \otimes L^{-2} \to \Lambda^2 \Omega^1_{\overline{X}}(\log D) \otimes L^{-3}={\mathcal O}_{\overline X}.
$$
Therefore, $E_{a,b}$ contains the vector bundle $S^{a+b} \Omega^1_{\overline{X}}(\log D) \otimes L^{-a-2b}$ as a direct summand.

Now we compute $H^1_{L^2}$ of the corresponding Higgs complex for $E_{a,b}$ in the same way as in the proof of theorem~\ref{vanishingtheorem}.
Ignoring again $L^2$--conditions first, then in degree one of the corresponding Higgs complex there is the vector bundle
$S^{a+b+1} \Omega^1_{\overline{X}}(\log D) \otimes  L^{-a-2b}$, which is in the kernel of $\theta$ but 
not killed by the differential $\theta$ from degree zero. This follows in the same way as in the proof of theorem~\ref{vanishingtheorem}.
Therefore the $H^0$ of this term survives in $H^1_{L^2}(\overline{X} ,E_{a,b})$.
For $a,b>0$ we however have $H^1_{L^2}(\overline{X} ,E_{a,b})=0$ by Corollary~\ref{vancor} and hence we have
$H^0(\overline{X} ,S^{a+b+1} \Omega_{\overline{X}}^1(\log D) \otimes {\mathcal O}_{\overline{X}}(-D) \otimes L^{-a-2b})=0$.
For all such choices of $a,b$ we let $n=a+b+1 \ge 3$ and $m=a+2b$ and we obtain the assertion for all possible values of $m \ge n \ge 3$ in this way.
\qed

\section{Higher dimensional ball quotients: Non--vanishing theorems}
\label{highervan}

If $\overline{X} \setminus D =\B_n =SU(n,1)/U(n)$ is an $n$--dimensional arithmetic ball quotient with smooth boundary $D$,
then the uniformizing Higgs bundle is as above $E_1=\left( \Omega_{\overline{X}}^1(\log D) \otimes L^{-1}\right) \oplus L^{-1}$
with $L^{n+1}={\mathcal O}_{\overline{X}}(K_{\overline{X}}+D)$. Its dual Higgs bundle is
$E_2=L \oplus \left( \Omega_{\overline{X}}^{n-1}(\log D) \otimes L^{-n} \right)$. The latter holds because of the perfect pairing
$\Omega_{\overline{X}}^1(\log D) \otimes \Omega_{\overline{X}}^{n-1}(\log D) \to  \Omega_{\overline{X}}^{n}(\log D)=L^{n+1}$.
The Higgs operator for $E_1$ is given by the identity map
$$
\theta: \Omega_{\overline{X}}^1(\log D) \otimes L^{-1} \to L^{-1} \otimes \Omega_{\overline{X}}^1(\log D)
$$
on $\Omega_{\overline{X}}^1(\log D) \otimes L^{-1}$ and by $0$ on $L^{-1}$.
Therefore we can compute its $L^2$--cohomology as above and obtain:
\begin{theorem}
$H^i_{L^2}(\overline{X},(E_1,\theta))$ is isomorphic to the $L^2$--sections inside
$$
H^0_{L^2}\left({\rm Ker}(\Omega_{\overline{X}}^1(\log D) \otimes L^{-1} \otimes
\Omega^i_{\overline{X}}(\log D) \to \Omega^{i+1}_{\overline{X}}(\log D) \otimes L^{-1})\right)
$$
for $i=1,\ldots,n$ and $H^0_{L^2}(\overline{X},(E_1,\theta)) \subseteq H^0(\overline{X},L^{-1})=0$.
\end{theorem}

\begin{proof}
The Higgs operator on each level is given by the canonical surjective map
$$
\Omega_{\overline{X}}^1(\log D) \otimes L^{-1} \otimes
\Omega^i_{\overline{X}}(\log D) \to \Omega^{i+1}_{\overline{X}}(\log D) \otimes L^{-1}
$$
on $ \Omega_{\overline{X}}^1(\log D) \otimes L^{-1} \otimes
\Omega^i_{\overline{X}}(\log D) $ and by $0$ on $L^{-1} \otimes
\Omega^i_{\overline{X}}(\log D)$. This proves the assertion.
\end{proof}


Turning to more general symmetric powers of $E_1=\left( \Omega_{\overline{X}}^1(\log D) \otimes L^{-1} \right) \oplus L^{-1}$,
we have the following result:
\begin{theorem} The $L^2$--Higgs complex for the symmetric power $(S^kE_1,\theta)$ with $k \ge 1$ is quasi--isomorphic to
$$
0 \to  T^0(k) {\buildrel 0 \over \to}  T^1(k) {\buildrel 0 \over \to} \to \cdots {\buildrel 0 \over \to}  T^n(k) \to 0,
$$
with trivial differentials, where $T^i(k)$ is the sheaf of $L^2$--sections of
$$
{\rm Ker}\left( S^k \Omega_{\overline{X}}^1(\log D) \otimes L^{-k} \otimes
\Omega^i_{\overline{X}}(\log D) \to S^{k-1} \Omega_{\overline{X}}^1(\log D) \otimes L^{-k}
\otimes \Omega^{i+1}_{\overline{X}}(\log D) \right)
$$
for $i=1,\ldots,n$. For $i=0$ we get $T_0(k)=L^{-k}$.\\
\end{theorem}

\begin{proof} The Higgs bundle associated to $S^k E_1$ is
$$
\left(S^k \Omega_{\overline{X}}^1(\log D) \otimes L^{-k}\right) \oplus
\left(S^{k-1}  \Omega_{\overline{X}}^1(\log D) \otimes L^{-k}\right) \oplus \cdots \oplus L^{-k}.
$$
If we write down the Higgs complex, then the Eagon--Northcott type complexes
$$
S^k \Omega_{\overline{X}}^1(\log D) \otimes L^{-k} \otimes
\Omega^i_{\overline{X}}(\log D) \to S^{k-1} \Omega_{\overline{X}}^1(\log D) \otimes L^{-k}
\otimes \Omega^{i+1}_{\overline{X}}(\log D) \to
$$
$$
\to S^{k-2} \Omega_{\overline{X}}^1(\log D) \otimes L^{-k}
\otimes \Omega^{i+2}_{\overline{X}}(\log D)
$$
occur which are exact in the middle \cite{green}. Hence the only cohomology arises at the left or right ends as stated.
\end{proof}

\begin{ex}
In the case $n=3$ this complex is
\begin{small}
$$
0 \to L^{-k} \to S^{k+1} \Omega_{\overline{X}}^1(\log D) \otimes L^{-k} \to
\Gamma_{k,1}(\Omega_{\overline{X}}^1(\log D)) \otimes L^{-k} \to
S^k \Omega_{\overline{X}}^1(\log D) \otimes L^{4-k} \to 0.
$$
\end{small}
Here $\Gamma_{a,b}$ is the standard irreducible representation
$$
\Gamma_{a,b}(W)={\rm Ker}\left(S^a(W) \otimes S^b(\Lambda^2 W)  \to S^{a-1}(W) \otimes S^{b-1}(\Lambda^2 W) \otimes \det(W)\right).
$$
associated to any $GL_3$--representation $W$.
\end{ex}

In a similar way we obtain a result for $E_2$:

\begin{theorem} \label{SKV2} The $L^2$--Higgs complex for the symmetric power $(S^kE_2,\theta)$ with $k \ge 1$ is quasi--isomorphic to
$$
0 \to  T^0(k) {\buildrel 0 \over \to}  T^1(k) {\buildrel 0 \over \to} \to \cdots {\buildrel 0 \over \to}  T^n(k) \to 0,
$$
with trivial differentials, where $T^i(k)$ is the sheaf of $L^2$--sections of
$$
{\rm Coker}\left( S^{k-1} \Omega_{\overline{X}}^{n-1}(\log D) \otimes L^{(n+1)-nk} \otimes
\Omega^{i-1}_{\overline{X}}(\log D) \to S^{k} \Omega_{\overline{X}}^{n-1}(\log D) \otimes L^{-nk}
\otimes \Omega^{i}_{\overline{X}}(\log D) \right)
$$
for $i=0,\ldots,n-1$. For $i=n$ we get $T^n(k)=L^{k+n+1}$.
\end{theorem}

\begin{proof} Argue as in the case of $E_1$.
\end{proof}

\begin{ex}
In the case $n=3$ this complex is
\begin{small}
$$
0 \to S^{k} \Omega_{\overline{X}}^{2}(\log D) \otimes L^{-3k} \to
\Gamma_{1,k}(\Omega_{\overline{X}}^1(\log D)) \otimes L^{-3k} \to
S^{k+1} \Omega_{\overline{X}}^2(\log D) \otimes L^{-3k} \to L^{k+4} \to 0.
$$
\end{small}
\end{ex}

\begin{cor} For $i=1$ and $k \ge 1$,
$H^1_{L^2}(\overline{X},(S^k E_1,\theta))$ is equal to the $L^2$--sections
of $H^0(\overline{X},S^{k+1} \Omega_{\overline{X}}^1(\log D) \otimes L^{-k})$.
This group is non--zero if $\Gamma$ is sufficiently small. In a similar way
$H^1_{L^2}(\overline{X},(S^k E_2,\theta))$ is equal to the $L^2$--sections
of $H^0(\Gamma_{1,\ldots,k}(\Omega_{\overline{X}}^1(\log D)) \otimes L^{-nk})$.
This group is non--zero if $\Gamma$ is sufficiently small.
\end{cor}
\begin{proof} A result of Kazdan \cite[page 255]{bw}, generalized in \cite{li}, implies that
$H^1_{L^2}(\overline{X},(S^k E_1,\theta))$ is non--zero if $\Gamma$ is sufficiently small.
\end{proof}

\section{The mixed Hodge structures on the
cohomology groups of certain local systems} \label{MHS}

In this section, we will discuss the mixed Hodge
structure on the cohomology groups of some local systems
underlying a polarized variation of Hodge structures over an (arithmetic) ball quotient surface.
Before that, we provide a short outline of the Deligne-Saito-Zucker theory of mixed
Hodge structures on cohomology groups with locally constant coefficients.

\subsection{General remarks about the mixed Hodge structure on the cohomology group of a local system} \ \\

Let us first recall some notation from the introduction and then 
introduce basic facts about mixed Hodge structures on the cohomology group of a local system.
Denote by $X$ a quasi-projective manifold of dimension $d$ and by
$(\W_{\R},\nabla,F^\cdot)$ a polarized $\R$-VHS over $X$ of weight $n$. Let
$\overline X$ be a smooth, projective compactification of $X$ such that
$D:=\overline X \setminus X$ is a simple normal crossing divisor. For simplicity
of exposition, we assume that the local monodromy around each
irreducible component of $D$ is unipotent (it is quasi-unipotent
in general in geometric situations). Put ${\mathcal W}=\W_{\R}\otimes_{\R}\sO_{X_{an}}$, where
$\sO_{X_{an}}$ is the sheaf of germs of holomorphic functions on
$X$. Deligne's canonical extension (see Ch II, \S4 in \cite{De0})
gives a unique extended vector bundle $\overline {\mathcal W}$ of ${\mathcal W}$
over $\overline X$, together with a flat logarithmic connection
$$
\overline \nabla: \overline{\mathcal W}\to \overline{\mathcal W}\otimes \Omega^1_{\overline{X}}(\log D).
$$
Using this we obtain the logarithmic de Rham complex $\Omega^*_{\log}(\overline{\mathcal W},\overline \nabla)$. Schmid's
Nilpotent orbit theorem implies that the Hodge filtration $F^\cdot$ extends to a filtration $\overline F^\cdot$ of holomorphic
subbundles of $\overline {\mathcal W}$ as well (see \S4 in \cite{Sch}). By
GAGA, the extended holomorphic objects over $\overline X$ are in fact
algebraic. One defines a Hodge filtration on the logarithmic
de Rham complex by
$$
F^r\Omega^*_{\log}(\overline{\mathcal W},\overline \nabla)=\Omega^*_{\overline X}(\log S)\otimes \overline F^{r-*},
$$
which is a subcomplex by Griffiths transversality. After Saito (see \cite{Sa1}-\cite{Sa4}), there is a naturally defined weight
filtration $W_\cdot$ on the logarithmic de Rham complex such that the triple $(\Omega^*_{\log}(\overline{\mathcal W},\overline
\nabla),W_\cdot,F^\cdot)$ is a cohomological mixed Hodge complex (see Appendix A in \cite{El1} and \cite{El2}). By Scholie 8.1.9 (ii) in \cite{De1}, this gives
rise to a real MHS with weights $\geq k+n$ on $H^k(X,\W_{\R})$. When $\W_\R$ is constant, this MHS coincides with the one defined
in \S3.2, \cite{De1} by Deligne. It is this MHS that we intend to understand properly in the case of a ball quotient surface.

\subsection{Case 1: $S^{k}\V_{1}\oplus S^{k}\V_2$ over ball quotient surfaces} \ \\

Let $X=\Gamma \backslash \B_2$ be an arithmetic ball quotient surface. Let $\V_1$ and $\V_2$ be the two standard complex local systems of rank $3$
over $X$. Their sum $\W_\R=\V_1 \oplus \V_2$ is a real local system.
As in the preceeding subsection there are two associated Deligne extensions
$\overline{\sV}_i$ on $\overline X$ (the toroidal compactification) such that $\overline{\mathcal W}=\overline{\sV}_1 \oplus \overline{\sV}_2$.
We denote the logarithmic de Rham complex for each $\overline{\sV}_i$ by $(\Omega^{\cdot}_{\overline{X}}(\log D) \otimes \overline{\sV}_i, \nabla)$.\\
In the following we explain the weight filtration on the logarithmic Higgs
complex of $S^{k}E_{1} \oplus S^{k}E_{2}$ which is the
logarithmic Higgs bundle obtained from $S^{k} \overline{\mathcal V}_{1} \oplus S^{k} \overline{\mathcal V}_{2}$
by taking graded pieces with respect to the Hodge filtration.
This Higgs bundle corresponds to the real local system $S^k \V_1 \oplus S^k \V_2$.
For the construction of the weight filtration on the logarithmic de Rham complex, one could
see \cite{Zuc2} in the case of noncompact curves and \cite{El1,El2} in
the general case. The weight filtration on the logarithmic Higgs
complex is then obtained by taking the graded pieces for the Hodge filtration on the
weight filtration of the logarithmic de Rham complex. In our approach we follow El Zein's work \cite{El1,El2}. In this reference all
technical details concerning the filtrations and the related spectral sequences are discussed. \\
Now we introduce some useful notation:

\begin{defi}
Let $I={i_1,...,i_j}, 0 \leq j\leq k$ be a subset of
$\{1,...,k\}$. Denote $\Omega^1_{\overline{X}}(\log D)^{\otimes
k}(I)$ the subsheaves of $\Omega^1_{\overline{X}}(\log D)^{\otimes
k}$ such that on the $i$th position ($i \in I$), the tensor factor
$\Omega^1_{\overline{X}}(\log D)$ is replaced by
$\Omega^1_{\overline{X}}$. For example,
$\Omega^1_{\overline{X}}(\log D)^{\otimes 3}(\{2,3\})=
\Omega^1_{\overline{X}}(\log D) \otimes \Omega^1_{\overline{X}}
\otimes \Omega^1_{\overline{X}}.$ Denote $
\Omega^1_{\overline{X}}(\log D)^{\otimes k}_j $ the subsheaves of
$\Omega^1_{\overline{X}}(\log D)^{\otimes k}$ given by
$$ \Sigma_{I,|I|=j}\Omega^1_{\overline{X}}(\log D)^{\otimes k}(I).$$
Denote $ S^{k} \Omega^1_{\overline{X}}(\log D)_{j} := S^{k}
\Omega^1_{\overline{X}}(\log D) \cap  \Omega^1_{\overline{X}}(\log
D)^{\otimes k}_j.$
\end{defi}

Using this notation, one can show:

\begin{prop}\label{weightcomplex1}
(a) The Higgs complex for $S^k E_1$ has the following shape:
The weight $(k+1)$--part $W_{k+1}(S^{k}E_{1}, \theta )$ is of the form
\begin{equation*}
\begin{CD}
Gr_{F}^{k+2}: & & Gr_F^{k+1}: & & Gr_F^{j} (0 \le j \le k): & & &
& 
\\ \hline
\\
&&&&...&& & &  
\\
&&&&\downarrow&&&& 
  \\
&& S^{k}\Omega^1_{\overline{X}}(\log D) \otimes L^{-k} \otimes
\Omega^1_{\overline{X}}(\log D) &&
... & &  & & 
\\
&& \downarrow  && \downarrow &&&& 
\\
S^{k}\Omega^1_{\overline{X}}(\log D) \otimes L^{-k} \otimes
\Omega^2_{\overline{X}}(\log D) & \quad &
S^{k-1}\Omega^1_{\overline{X}}(\log D) \otimes L^{-k} \otimes
\Omega^2_{\overline{X}}(\log D) & \quad & ... & &  && 
 \\ \\  \hline
\end{CD}
\end{equation*} \ \\
where the column under $ Gr_F^{j} (0 \le j \le k)$ is given by
$$
S^j \Omega_{\overline{X}}^1(\log D) \otimes L^{-k} 
\to S^{j-1} \Omega_{\overline{X}}^1(\log D) \otimes L^{-k} \otimes
\Omega^{1}_{\overline{X}}(\log D)
$$
$$
\to S^{j-2} \Omega_{\overline{X}}^1(\log D) \otimes L^{-k} \otimes
\Omega^{2}_{\overline{X}}(\log D).
$$
(As a notation, $S^{j} \Omega_{\overline{X}}^1(\log D)$ means $0$
if $j<0$) \\
 For every $1\leq l \leq k$ the weight $l$--part
$W_{l}(S^{k}E_{1}, \theta )$   is of the form
\begin{equation*}
\begin{CD}
Gr_{F}^{k+2}: &  & Gr_F^{k+1} : & & Gr_F^{j} (0 \le j \le k): & &
& & 
 \\
 \hline
 \\
&&&& ...&&  &&  
\\
&&&& \downarrow && && 
\\
&& S^{k}\Omega^1_{\overline{X}}(\log D) \otimes L^{-k} \otimes
\Omega^1_{\overline{X}} + &&... &  &  && 
\\
&&  S^{k}\Omega^1_{\overline{X}}(\log D)_{k-l+1} \otimes L^{-k}
\otimes \Omega^1_{\overline{X}}(\log D)
\\ && \downarrow&&\downarrow  &&&& 
\\
S^{k}\Omega^1_{\overline{X}}(\log D) \otimes L^{-k} \otimes
\Omega^2_{\overline{X}} +& & S^{k-1}\Omega^1_{\overline{X}}(\log
D) \otimes L^{-k} \otimes
\Omega^2_{\overline{X}}(\log D) &&...& &  && 
\\
 S^{k}\Omega^1_{\overline{X}}(\log D)_{k-l+1} \otimes L^{-k}
\otimes \Omega^2_{\overline{X}}(\log D) \\ \\ \hline
\end{CD}
\end{equation*} \ \\

where the column under $ Gr_F^{j} (0 \le j \le k)$ is given by
$$
S^j \Omega_{\overline{X}}^1(\log D) \otimes L^{-k} 
\to S^{j-1} \Omega_{\overline{X}}^1(\log D) \otimes L^{-k} \otimes
\Omega^{1}_{\overline{X}}(\log D)
$$
$$
\to S^{j-2} \Omega_{\overline{X}}^1(\log D) \otimes L^{-k} \otimes
\Omega^{2}_{\overline{X}}(\log D).
$$

The weight $0$--part $W_{0}(S^{k}E_{1}, \theta )$ is of the form

\begin{equation*}
\begin{CD}
Gr_{F}^{k+2}: &  &Gr_{F}^{k+1}:&& && Gr_F^{j} (0 \le j \le k): & &
&
& 
\\
\hline
\\
&&&&   &&  ...
\\
&&&&&& \downarrow
\\
&& S^{k}\Omega^1_{\overline{X}}(\log D) \otimes L^{-k} \otimes
\Omega^1_{\overline{X}} &&  &&...\\
&& \downarrow  &&&&  \downarrow   \\
S^{k}\Omega^1_{\overline{X}}(\log D) \otimes L^{-k} \otimes
\Omega^2_{\overline{X}} & \quad &
S^{k-1}\Omega^1_{\overline{X}}(\log D) \otimes L^{-k} \otimes
\Omega^2_{\overline{X}}(\log D) &&  && ... \\  \\ \hline
\end{CD}
\end{equation*} \ \\

where the column under $ Gr_F^{j} (0 \le j \le k)$ is given by
$$
S^j \Omega_{\overline{X}}^1(\log D) \otimes L^{-k} 
\to S^{j-1} \Omega_{\overline{X}}^1(\log D) \otimes L^{-k} \otimes
\Omega^{1}_{\overline{X}}(\log D)
$$
$$
\to S^{j-2} \Omega_{\overline{X}}^1(\log D) \otimes L^{-k} \otimes
\Omega^{2}_{\overline{X}}(\log D).
$$

(b) The Higgs complex for $S^k E_2$ has the following shape: The
weight $(k+1)$--part $W_{k+1}(S^{k}E_{2}, \theta )$ is of the form
\begin{equation*}
\begin{CD}
Gr_F^{j} (2 \le j \le k+2): &  &  &  & Gr_F^{1} : & & Gr_{F}^{0}:
\\ \hline
\\
...&& &  &  S^{k-1}\Omega^1_{\overline{X}}(\log D) \otimes L^{3-2k}    & &  S^{k}\Omega^1_{\overline{X}}(\log D) \otimes L^{-2k}  \\
\downarrow &&&& \downarrow  && \\
...&&   && S^{k}\Omega^1_{\overline{X}}(\log D) \otimes L^{-2k}
\otimes \Omega^1_{\overline{X}}(\log D)
\\ \downarrow   &&&& \\ 
...&& \\
\\
\hline
\end{CD}
\end{equation*} \ \\

where the column under $ Gr_F^{j} (2 \le j \le k+2)$ is given by
$$
 S^{k-j} \Omega_{\overline{X}}^{1}(\log D)
\otimes L^{3j-2k} 
\to S^{k-j+1} \Omega_{\overline{X}}^{1}(\log D) \otimes
L^{3j-3-2k} \otimes \Omega^{1}_{\overline{X}}(\log D)
$$
$$
\to S^{k-j+2} \Omega_{\overline{X}}^{1}(\log D) \otimes
L^{3j-6-2k} \otimes \Omega^{2}_{\overline{X}}(\log D).
$$

For every $1\leq l \leq k$ the weight $l$--part $W_{l}(S^{k}E_{2},
\theta )$   is of the form
\begin{equation*}
\begin{CD}
Gr_F^{j} (2 \le j \le k+2): &  &  &  & Gr_F^{1} : & & Gr_{F}^{0}:
\\ \hline
\\
...&&&  &  S^{k-1}\Omega^1_{\overline{X}}(\log D) \otimes L^{3-2k}    && S^{k}\Omega^1_{\overline{X}}(\log D) \otimes L^{-2k}  \\
\downarrow &&&& \downarrow  &&  \\
...&&  & & S^{k}\Omega^1_{\overline{X}}(\log D) \otimes
L^{-2k} \otimes \Omega^1_{\overline{X}}(\log D) \\
\downarrow  &&&&&&   \\
... \\
\\
\hline
\end{CD}
\end{equation*} \ \\

where the column under $ Gr_F^{j} (2 \le j \le l)$ is given by
$$
 S^{k-j} \Omega_{\overline{X}}^{1}(\log D)
\otimes L^{3j-2k} 
\to S^{k-j+1} \Omega_{\overline{X}}^{1}(\log D) \otimes
L^{3j-3-2k} \otimes \Omega^{1}_{\overline{X}}(\log D)
$$
$$
\to S^{k-j+2} \Omega_{\overline{X}}^{1}(\log D) \otimes
L^{3j-6-2k} \otimes \Omega^{2}_{\overline{X}}(\log D),
$$
the column under $ Gr_F^{l+1} $ is given by
$$
 S^{k-l-1} \Omega_{\overline{X}}^{1}(\log D)
\otimes L^{3l+3-2k} 
\to S^{k-l}\Omega^1_{\overline{X}}(\log D) \otimes L^{3l-2k}
\otimes \Omega^1_{\overline{X}}
$$
$$
+ S^{k-l}\Omega^1_{\overline{X}}(\log D)_{1} \otimes L^{3l-2k}
\otimes \Omega^1_{\overline{X}}(\log D) \to S^{k-l+1}
\Omega_{\overline{X}}^{1}(\log D) \otimes L^{3l-3-2k} \otimes
\Omega^{2}_{\overline{X}}(\log D),
$$
and the column under $Gr_F^{j} (l+2 \le j \le k+2)$ is given by:
$$
 S^{k-j} \Omega_{\overline{X}}^{1}(\log D)
\otimes L^{3j-2k} 
\to S^{k-j+1}\Omega^1_{\overline{X}}(\log D) \otimes
L^{3j-3-2k}\otimes \Omega^1_{\overline{X}}
$$
$$
+ S^{k-j+1}\Omega^1_{\overline{X}}(\log D)_{1} \otimes L^{3j-3-2k}
\otimes \Omega^1_{\overline{X}}(\log D) \to S^{k-j+2}
\Omega_{\overline{X}}^{1}(\log D) \otimes L^{3j-6-2k} \otimes
\Omega^{2}_{\overline{X}}.
$$

The weight $0$--part $W_{0}(S^{k}E_{2}, \theta )$ is of the form
\begin{equation*}
\begin{CD}
 Gr_F^{j} (2 \le j \le k+2): &  &  &  & Gr_F^{1} : &  & Gr_{F}^{0}: \\ \hline
\\
...&& &  &  S^{k-1}\Omega^1_{\overline{X}}(\log D) \otimes L^{3-2k}    & & S^{k}\Omega^1_{\overline{X}}(\log D) \otimes L^{-2k}  \\
\downarrow &&&& \downarrow  &&  \\
...&&  & & S^{k}\Omega^1_{\overline{X}}(\log D) \otimes L^{-2k}
\otimes \Omega^1_{\overline{X}} \\
&& && +  S^{k}\Omega^1_{\overline{X}}(\log D)_{1} \otimes L^{-2k}
\otimes \Omega^1_{\overline{X}}(\log D) \\
\downarrow  &&&&&&  \\
... && \\ \\ \hline
\end{CD}
\end{equation*}

where the column under $Gr_F^{j} (2 \le j \le k+2)$ is given by
$$
 S^{k-j} \Omega_{\overline{X}}^{1}(\log D)
\otimes L^{3j-2k} 
\to S^{k-j+1}\Omega^1_{\overline{X}}(\log D) \otimes
L^{3j-3-2k}\otimes \Omega^1_{\overline{X}}
$$
$$
+ S^{k-j+1}\Omega^1_{\overline{X}}(\log D)_{1} \otimes L^{3j-3-2k}
\otimes \Omega^1_{\overline{X}}(\log D) \to S^{k-j+2}
\Omega_{\overline{X}}^{1}(\log D) \otimes L^{3j-6-2k} \otimes
\Omega^{2}_{\overline{X}}.
$$

\end{prop}

\begin{proof}
This can be worked out using a recent result of El Zein \cite[\S 4]{El1} for the special case where the boundary divisor $D$ is smooth.
\end{proof}


Using these results, we can prove the following theorem:

\begin{theorem}
The mixed Hodge numbers of $H^{l}(X, S^{k}\V_{1}\oplus
S^{k}\V_{2}), 0\leq l \leq 4$ above weight $l+k$ depend only on
the number of boundary components $h$. More precisely: \\
1. For $l=3$, the only non--zero mixed Hodge number is
$h^{k+2,k+2}_{3}=2h.$ \\
2. For $l=2$, the only non--zero mixed Hodge numbers are
$h^{k+2,1}_{2}=h^{1,k+2}_{2}\leq h$ and
$h^{k+2,k+1}_{2}=h^{k+1,k+2}_{2}= h$. \\
Furthermore, if
$$
H^{1}(X, S^{k}\V_{1}\oplus S^{k}\V_{2})=0,
$$
then $h^{k+2,1}_{2}=h^{1,k+2}_{2} = h.$\\
3. For $l=1$ one has
$$
H^{1}(X, S^{k}\V_{1}\oplus S^{k}\V_{2}) = W_{0}H^{1}(X, S^{k}\V_{1}\oplus S^{k}\V_{2}).
$$
\end{theorem}

\begin{proof}  Let $\V:= S^{k}\V_{1}\oplus S^{k}\V_{2}$. Note that:
$$ H^{l}(X, \V)=0 \text{ for } l=0,4,
$$
so we only need to consider the case when $l=1,2,3$. Let
$$ h^{p,q}_{l} := \dim
Gr_{F}^{p}Gr_{\overline{F}}^{q}Gr^{W[l+k]}_{p+q}(H^{k}(X,\V)),$$
$$ h^{p,q}_{l,1} := \dim
Gr_{F}^{p}Gr_{\overline{F}}^{q}Gr^{W[l+k]}_{p+q}(H^{k}(X,S^{k}\V_{1})),$$
$$ h^{p,q}_{l,2} := \dim
Gr_{F}^{p}Gr_{\overline{F}}^{q}Gr^{W[l+k]}_{p+q}(H^{k}(X,S^{k}\V_{2})).$$
Note that $ h^{p,q}_{l}=  h^{p,q}_{l,1}+ h^{p,q}_{l,2}$ and $
h^{q,p}_{l,1}= h^{p,q}_{l,2} $.
From Prop. \ref{weightcomplex1} we deduce that $ h^{p,q}_{l} =0$ for $ p+q
> l+  2k+1 $ or $ p+q < l+k $.
From the calculation of the cohomology of the logarithmic Higgs complex
we conclude that $h^{p,q}_{2,1}=0$ for $0 < p <k+1, p+q > k+2$ and
$h^{p,q}_{3,1}=0$ for  $0 < p <k+2, p+q > k+2$. On the other
hand, by theorem \ref{SKV2}, we know that
$h^{p,q}_{2,2}=0$ for $1 < p <k+2, p+q > k+2$ and $h^{p,q}_{3,2}=0$ for $0 < p <k+2, p+q > k+2$. \\
{\bf Case 1.} $l=1$ \\
Let $X^{*}$ be the Baily-Borel compactification of $X$, and note
that we have the following commutative diagram: \\
\xymatrix{ IH^{1}(X^{*},S^{k}\V_{1}\oplus S^{k}\V_{2} ) \ar[r]
\ar[rd] & IH^{1}(\overline{X},
S^{k}\V_{1}\oplus S^{k}\V_{2}) \ar[d]\\
 &
H^{1}(X, S^{k}\V_{1}\oplus S^{k}\V_{2}).
}\\
Since the singularities of $X^{*}$ are isolated points, we know
that the maps $$IH^{1}(X^{*},S^{k}\V_{1}\oplus S^{k}\V_{2} )
\rightarrow H^{1}(X, S^{k}\V_{1}\oplus S^{k}\V_{2}) $$ are
isomorphisms \cite{AlanH}. Therefore, the map
$$IH^{1}(\overline{X},S^{k}\V_{1}\oplus S^{k}\V_{2} ) \rightarrow
H^{1}(X, S^{k}\V_{1}\oplus S^{k}\V_{2}) $$ is surjective. So we
have $$ H^{1}(X, S^{k}\V_{1}\oplus S^{k}\V_{2}) = W_{0}H^{1}(X,
S^{k}\V_{1}\oplus S^{k}\V_{2}). $$ \noindent
{\bf Case 2.} $l=3$ \\
By Hodge symmetry we have $h^{k+2,q}_{3}= h^{q,k+2}_{3}=0, 0 \leq q \leq k+1$.
So the only unknown Hodge number is $h^{k+2,k+2}_{3}.$ \\
By $E_2$-degeneration of the weight filtration,
$Gr_{F}^{k+2}Gr_{\overline{F}}^{k+2}Gr^{W[3+k]}_{k+2+k+2}(H^{3}(X,\V))=$
$$\frac{{\rm Ker}(
H^{3}(\overline{X},
Gr^{W}_{k+1}Gr_{F}^{k+2}(\Omega^{\cdot}_{\overline{X}}(\log D)
\otimes \overline{\sV}, \nabla)) \rightarrow H^{4}(\overline{X},
Gr^{W}_{k}Gr_{F}^{k+2}(\Omega^{\cdot}_{\overline{X}}(\log D)
\otimes \overline{\sV}, \nabla)))}{{\rm Image}( H^{2}(\overline{X},
Gr^{W}_{k+2}Gr_{F}^{k+2}(\Omega^{\cdot}_{\overline{X}}(\log D)
\otimes \overline{\sV}, \nabla)) \rightarrow H^{3}(\overline{X},
Gr^{W}_{k+1}Gr_{F}^{k+2}(\Omega^{\cdot}_{\overline{X}}(\log D)
\otimes \overline{\sV}, \nabla)))}.
$$
Since
$$
H^{4}(\overline{X}, Gr^{W}_{k}Gr_{F}^{k+2}(\Omega^{\cdot}_{\overline{X}}(\log D)
\otimes \overline{\sV}, \nabla)))= H^{2}(\overline{X},
Gr^{W}_{k+2}Gr_{F}^{k+2}(\Omega^{\cdot}_{\overline{X}}(\log D)
\otimes \overline{\sV}, \nabla)) =0,
$$
we get
$$
Gr_{F}^{k+2}Gr_{\overline{F}}^{k+2}Gr^{W[3+k]}_{k+2+k+2}(H^{3}(X,\V)) = H^{3}(\overline{X},
Gr^{W}_{k+1}Gr_{F}^{k+2}(\Omega^{\cdot}_{\overline{X}}(\log D)
\otimes \overline{\sV}, \nabla)))$$ $$=H^{1}(D, \sO_D) \oplus H^{1}(D,
\sO_D) .
$$
Therefore $h^{k+2,k+2}_{3}=2h$, where $h$ is the number
of connected components contained in $D$. \\
\noindent {\bf Case 3.} $l=2$ \\
By Hodge symmetry we have $h^{p,q}_{2}= h^{q,p}_{2}=0$ for $1 < q < k+1$. 
So the only unknown Hodge numbers are $h^{k+2,k+1}_{2}= h^{k+1,k+2}_{2}, h^{k+1,k+1}_{2}, h^{k+2,1}_{2}=h^{1,k+2}_{2}.$ \\
For $h^{k+2,k+1}_{2}$, we know that $h^{k+2,k+1}_{2,1}=
h^{k+1,k+2}_{2,2}=0$. Hence we only need to compute
$h^{k+2,k+1}_{2,2}$.  Denote by $\overline\sV_{i}, i=1,2 $ the Deligne
extension of $S^{k}\V_{i}$. We have
$Gr_{F}^{k+1}Gr_{\overline{F}}^{k+2}Gr^{W[2+k]}_{k+1+k+2}(H^{2}(X,S^{k}\V_{2}))
= $
$$\frac{{\rm Ker}(
H^{2}(\overline{X},
Gr^{W}_{k+1}Gr_{F}^{k+1}(\Omega^{\cdot}_{\overline{X}}(\log D)
\otimes \overline{\sV_{2}}, \nabla)) \rightarrow H^{3}(\overline{X},
Gr^{W}_{k}Gr_{F}^{k+1}(\Omega^{\cdot}_{\overline{X}}(\log D)
\otimes \overline{\sV_{2}}, \nabla)))}{{\rm Image}( H^{1}(\overline{X},
Gr^{W}_{k+2}Gr_{F}^{k+1}(\Omega^{\cdot}_{\overline{X}}(\log D)
\otimes \overline{\sV_{2}}, \nabla)) \rightarrow H^{2}(\overline{X},
Gr^{W}_{k+1}Gr_{F}^{k+1}(\Omega^{\cdot}_{\overline{X}}(\log D)
\otimes \overline{\sV_{2}}, \nabla)))}. $$
Furthermore,
$$
H^{3}(\overline{X},
Gr^{W}_{k}Gr_{F}^{k+1}(\Omega^{\cdot}_{\overline{X}}(\log D)
\otimes \overline{\sV_{2}}, \nabla))= H^{1}(\overline{X},
Gr^{W}_{k+2}Gr_{F}^{k+1}(\Omega^{\cdot}_{\overline{X}}(\log D)
\otimes \overline{\sV_{2}}, \nabla)) =0,
$$
so we get
$$Gr_{F}^{k+1}Gr_{\overline{F}}^{k+2}Gr^{W[2+k]}_{k+2+k+1}(H^{2}(X,S^{k}\V_{2})) =
H^{2}(\overline{X},
Gr^{W}_{k+1}Gr_{F}^{k+1}(\Omega^{\cdot}_{\overline{X}}(\log D)
\otimes \overline{\sV_{2}}, \nabla))= H^{1}(D, \sO_D). $$ Therefore
$h^{k+2,k+1}_{2,2}=h^{k+2,k+1}_{2}= h^{k+1,k+2}_{2} = h$. \\

Since $h^{k+1,k+1}_{2,2}=0$, one has $h^{k+1,k+1}_{2,1}=0$. Therefore
$h^{k+1,k+1}_{2}=0$.\\

One has
$Gr_{F}^{k+2}Gr_{\overline{F}}^{1}Gr^{W[2+k]}_{k+2+1}(H^{2}(X,\V)) = $
$$\frac{{\rm Ker}(
H^{2}(\overline{X},
Gr^{W}_{1}Gr_{F}^{k+2}(\Omega^{\cdot}_{\overline{X}}(\log D)
\otimes \overline{\sV}, \nabla)) \rightarrow H^{3}(\overline{X},
Gr^{W}_{0}Gr_{F}^{k+2}(\Omega^{\cdot}_{\overline{X}}(\log D)
\otimes \overline{\sV}, \nabla)))}{{\rm Image}( H^{1}(\overline{X},
Gr^{W}_{2}Gr_{F}^{k+2}(\Omega^{\cdot}_{\overline{X}}(\log D)
\otimes \overline{\sV}, \nabla)) \rightarrow H^{2}(\overline{X},
Gr^{W}_{1}Gr_{F}^{k+2}(\Omega^{\cdot}_{\overline{X}}(\log D)
\otimes \overline{\sV}, \nabla)))}. $$ Since $H^{1}(\overline{X},
Gr^{W}_{2}Gr_{F}^{k+2}(\Omega^{\cdot}_{\overline{X}}(\log D)
\otimes \overline{\sV}, \nabla))= 0$,\\
 we have
$$
Gr_{F}^{k+2}Gr_{\overline{F}}^{1}Gr^{W[2+k]}_{k+2+1}(H^{2}(X,\V))
\subset
Gr^{W}_{1}Gr_{F}^{k+2}(\Omega^{\cdot}_{\overline{X}}(\log D)
\otimes \overline{\sV}, \nabla))= H^{0}(D, \sO_D).
$$
Therefore
$h^{k+2,1}_{2}=h^{1,k+2}_{2}\leq h$.
\end{proof}

The inequality in the theorem becomes an equality under the following assumption:

\begin{cor}
If we assume
$$ IH^{3}(X^{*},S^{k}\V_{1}\oplus S^{k}\V_{2} )^{k+2,1} \cong
W_{0}H^{3}(X, S^{k}\V_{1}\oplus S^{k}\V_{2})^{k+2,1}, $$ where the
index $(k+2,1)$ denotes the $(k+2,1)$ Hodge component, then we have
$$
h^{k+2,1}_{2}=h^{1,k+2}_{2}= h.
$$
Note that if
$$
H^{1}(X, S^{k}\V_{1}\oplus S^{k}\V_{2})=0,
$$
then
$$
IH^{3}(X^{*},S^{k}\V_{1}\oplus S^{k}\V_{2} ) \cong IH^{1}(X^{*},S^{k}\V_{1}\oplus S^{k}\V_{2}
) =0,
$$
and hence this assumption is satisfied, since by \cite{ps} this implies
$$
IH^{3}(X^{*},S^{k}\V_{1}\oplus S^{k}\V_{2} )^{k+2,1} \cong
W_{0}H^{3}(X, S^{k}\V_{1}\oplus S^{k}\V_{2})^{k+2,1}=0.
$$
\end{cor}

\begin{proof} From the Higgs complex and the proof of the theorem it follows that
$$
\dim IH^{1}(\overline{X},S^{k}\V_{1}\oplus S^{k}\V_{2} )^{0,k+1}=\dim W_{0}H^{1}(X, S^{k}\V_{1}\oplus S^{k}\V_{2})^{0,k+1}.
$$
Using the assumption (and \cite{ps}) we obtain therefore
$$
\dim IH^{3}(X^{*},S^{k}\V_{1}\oplus S^{k}\V_{2} )^{k+2,1}= \dim W_{0}H^{3}(X, S^{k}\V_{1}\oplus
S^{k}\V_{2})^{k+2,1}
$$
$$
\leq \dim IH^{3}(\overline{X},S^{k}\V_{1}\oplus S^{k}\V_{2} )^{k+2,1} =
\dim IH^{1}(\overline{X},S^{k}\V_{1}\oplus S^{k}\V_{2} )^{0,k+1}
$$
$$
= \dim W_{0}H^{1}(X, S^{k}\V_{1}\oplus S^{k}\V_{2})^{0,k+1} =
\dim IH^{1}(X^{*},S^{k}\V_{1}\oplus S^{k}\V_{2} )^{0,k+1}.
$$
Since by duality
$$
\dim IH^{3}(X^{*},S^{k}\V_{1}\oplus S^{k}\V_{2} )^{k+2,1}
= \dim IH^{1}(X^{*},S^{k}\V_{1}\oplus S^{k}\V_{2} )^{0,k+1},
$$
we have equality in the above inequality and hence
$$
\dim W_{0}H^{3}(X, S^{k}\V_{1}\oplus S^{k}\V_{2})^{k+2,1}=
\dim IH^{3}(\overline{X},S^{k}\V_{1}\oplus S^{k}\V_{2} )^{k+2,1}.
$$
Then the map
 $$H^{2}(\overline{X},
Gr^{W}_{1}Gr_{F}^{k+2}(\Omega^{\cdot}_{\overline{X}}(\log D)
\otimes \overline{\sV}, \nabla)) \rightarrow H^{3}(\overline{X},
Gr^{W}_{0}Gr_{F}^{k+2}(\Omega^{\cdot}_{\overline{X}}(\log D)
\otimes \overline{\sV}, \nabla)))$$ is trivial.
 We have
$$Gr_{F}^{k+2}Gr_{\overline{F}}^{1}Gr^{W[2+k]}_{k+2+1}(H^{2}(X,\V))
\cong
Gr^{W}_{1}Gr_{F}^{k+2}(\Omega^{\cdot}_{\overline{X}}(\log D)
\otimes \overline{\sV}, \nabla))= H^{0}(D, \sO_D). $$ Therefore
$h^{k+2,1}_{2}=h^{1,k+2}_{2}= h$.
\end{proof}

\subsection{Case 2: ${\rm End}^0(\V_1)={\rm End}^0(\V_2)$ over ball
quotient surfaces} \ \ \\

\begin{theorem}
The mixed Hodge numbers of $H^{l}(X, {\rm End}^0(\V_1)), 0\leq l
\leq 4$ above weight $l+2$ depend only on the
number of boundary components $h$. More precisely: \\
1. For $l=3$, the nonzero mixed Hodge number is only
$h^{4,4}_{3}=h.$ \\
2. For $l=2$, the nonzero mixed Hodge numbers are only
$h^{4,2}_{2}=h^{2,4}_{2}= h.$ \\
\end{theorem}

\begin{proof} We proceed as in case 1, and first compute the weight filtration
on the logarithmic Higgs complex of ${\rm End}^0(E_1)= {\rm End}^0(E_2) \subset E_{1} \otimes E_{2}$ which
is the logarithmic Higgs bundle associated to the Deligne extension of $\V:={\rm End}^0(\V_1)={\rm End}^0(\V_2)$.
\\

$W_{3}({\rm End}^0(E_1), \theta ):$ \\
\begin{tiny}
\begin{equation*}
\begin{CD}
Gr_{F}^{4 }: && Gr_{F}^{3 }: && Gr_{F}^{2 }: && Gr_{F}^{1 }: &&
Gr_{F}^{0 }: \\ \hline
 &&&& \Omega^1_{\overline{X}}(\log D) &\oplus &
\Omega^1_{\overline{X}}(\log
D) ^{\otimes 2} \otimes L^{-3}  &\oplus& \Omega^1_{\overline{X}}(\log D) \otimes L^{-3}  \\
&&&& \downarrow &&\downarrow \cong &&  \\
&&\Omega^1_{\overline{X}}(\log D)^{\otimes 2} &\oplus &
\Omega^1_{\overline{X}}(\log D)^{ \otimes 3} \otimes L^{-3}
&\oplus& \Omega^1_{\overline{X}}(\log D)^{\otimes 2} \otimes
L^{-3}
\\ && \downarrow \cong && \downarrow  &&   \\  \Omega^1_{\overline{X}}(\log D)\otimes L^3 &\oplus& \Omega^1_{\overline{X}}(\log D)^{\otimes
2}  &\oplus&
\Omega^1_{\overline{X}}(\log D)  \\ \hline
\end{CD}
\end{equation*}
\end{tiny}
$W_{2}({\rm End}^0(E_1), \theta ):$ \\

\begin{tiny}
\begin{equation*}
\begin{CD}
Gr_{F}^{4 }: && Gr_{F}^{3 }: && Gr_{F}^{2 }: && Gr_{F}^{1 }: &&
Gr_{F}^{0 }: \\ \hline
&&&& \Omega^1_{\overline{X}}(\log D) &\oplus &
\Omega^1_{\overline{X}}(\log
D) ^{\otimes 2} \otimes L^{-3}  &\oplus& \Omega^1_{\overline{X}}(\log D) \otimes L^{-3}  \\
&&&& \downarrow &&\downarrow \cong &&  \\ &&
\Omega^1_{\overline{X}} \otimes \Omega^1_{\overline{X}}(\log D) +
\Omega^1_{\overline{X}}(\log D) \otimes \Omega^1_{\overline{X}}
&\oplus & \Omega^1_{\overline{X}}(\log D)^{ \otimes 3} \otimes
L^{-3} &\oplus& \Omega^1_{\overline{X}}(\log D)^{\otimes
2} \otimes L^{-3}  \\
 && \downarrow  && \downarrow  &&   \\
\Omega^1_{\overline{X}}\otimes L^3 &\oplus&
\Omega^1_{\overline{X}}(\log D)^{\otimes 2}  &\oplus&
\Omega^1_{\overline{X}}(\log D)  \\ \hline
\end{CD}
\end{equation*}
\end{tiny}

$W_{1}({\rm End}^0(E_1), \theta ):$ \\

\begin{tiny}
\begin{equation*}
\begin{CD}
Gr_{F}^{4 }: && Gr_{F}^{3 }: && Gr_{F}^{2 }: && Gr_{F}^{1 }: &&
Gr_{F}^{0 }: \\ \hline
&&&& \Omega^1_{\overline{X}}(\log D) &\oplus &
\Omega^1_{\overline{X}}(\log
D) ^{\otimes 2} \otimes L^{-3}  &\oplus& \Omega^1_{\overline{X}}(\log D) \otimes L^{-3} \\
 &&&& \downarrow &&\downarrow \cong &&  \\ &&
\Omega^1_{\overline{X}}(\log D) \otimes \Omega^1_{\overline{X}}
&\oplus & ( \Omega^1_{\overline{X}}(\log D) \otimes
\Omega^1_{\overline{X}} \otimes \Omega^1_{\overline{X}}(\log D)
&\oplus&
 \Omega^1_{\overline{X}}(\log
D)^{\otimes
2} \otimes L^{-3}  \\
 &&&& + \Omega^1_{\overline{X}}\otimes
\Omega^1_{\overline{X}}(\log D)^{\otimes 2} \\
&&&&  + \Omega^1_{\overline{X}}(\log D)^{\otimes 2} \otimes \Omega^1_{\overline{X}}) \otimes L^{-3} \\
 && \downarrow  && \downarrow  &&   \\
\Omega^1_{\overline{X}}(\log D) \otimes \Omega^2_{\overline{X}}
&\oplus&
  \Omega^1_{\overline{X}} \otimes  \Omega^1_{\overline{X}}(\log D) &\oplus&
\Omega^1_{\overline{X}}(\log D)   \\
&& + \Omega^1_{\overline{X}}(\log D) \otimes
\Omega^1_{\overline{X}} \\ \hline
\end{CD}
\end{equation*}
\end{tiny}

$W_{0}({\rm End}^0(E_1), \theta ):$ \\
\begin{tiny}
\begin{equation*}
\begin{CD}
Gr_{F}^{4 }: && Gr_{F}^{3 }: && Gr_{F}^{2 }: && Gr_{F}^{1 }: &&
Gr_{F}^{0 }: \\
&&&& \Omega^1_{\overline{X}}(\log D) &\oplus &
\Omega^1_{\overline{X}}(\log
D) ^{\otimes 2}  \otimes L^{-3}  &\oplus& \Omega^1_{\overline{X}}(\log D)  \otimes L^{-3}  \\
&&&& \downarrow &&\downarrow \cong &&  \\
&& \Omega^1_{\overline{X}}(\log D) \otimes \Omega^1_{\overline{X}}
&\oplus & (\Omega^1_{\overline{X}}(\log D) \otimes
\Omega^1_{\overline{X}} \otimes \Omega^1_{\overline{X}}(\log D)
&\oplus& \Omega^1_{\overline{X}}(\log D)^{\otimes
2}\otimes L^{-3}  \\
&&&&+\Omega^1_{\overline{X}}(\log D)^{ \otimes
2} \otimes \Omega^1_{\overline{X}})\otimes L^{-3} && \\
 && \downarrow  \cong && \downarrow  &&   \\
\Omega^1_{\overline{X}}(\log D)  \otimes  \Omega^2_{\overline{X}}
 &\oplus&
  \Omega^1_{\overline{X}} (\log D) \otimes
 \Omega^1_{\overline{X}}  &\oplus&
\Omega^1_{\overline{X}}(\log D)  \\  \hline
\end{CD}
\end{equation*}
\end{tiny}

Note that:
$$ H^{l}(X, \V)=0, \text{ for } l=0,1,4,$$
so we only need to consider the case when $l=2,3$. Let
$$ h^{p,q}_{l} := \dim
Gr_{F}^{p}Gr_{\overline{F}}^{q}Gr^{W[l+2]}_{p+q}(H^{l}(X,\V)).$$ From
the above calculation, we know that $ h^{p,q}_{l} =0$ for $ p+q
> l+2+3 $ or $ p+q < l+2 $.\\
Denote the logarithmic de Rham complex of $\overline{\sV}$ by
$(\Omega^{\cdot}_{\overline{X}}(\log D) \otimes \overline{\sV}, \nabla)$.\\

1. $l=3$ \\
We use the calculation of the weight filtration on the logarithmic
Higgs complex in Prop.~\ref{weightcomplex1} together with the
vanishing result Cor.~\ref{vancor}. This implies the exactness of
the graded subquotients $Gr_{F}^{3}$ and $Gr_{F}^{1}$ of
$W_{3}({\rm End}^0(E_1), \theta )$, and hence that all the Hodge
numbers $h^{3,5}_{3}$, $h^{6,2}_{3}$, $h^{1,7}_{3}$,
$h^{8,0}_{3}$, $h^{3,4}_{3}$, $h^{5,2}_{3}$, $h^{1,6}_{3}$,
$h^{7,0}_{3}$, $h^{3,3}_{3}$, $h^{4,2}_{3}$, $h^{1,5}_{3}$ and
$h^{6,0}_{3}$ are $0$. Consequently, by Hodge
symmetry, we only need to compute $h^{4,4}_{3}$. \\
By $E_2$-degeneration of the weight filtration, $
Gr_{F}^{4}Gr_{\overline{F}}^{4}Gr^{W[3+2]}_{4+4}(H^{3}(X,\V)) = $
 $$\frac{{\rm Ker}(
H^{3}(\overline{X},
Gr^{W}_{3}Gr_{F}^{4}(\Omega^{\cdot}_{\overline{X}}(\log D) \otimes
\overline{\sV}, \nabla)) \rightarrow H^{4}(\overline{X},
Gr^{W}_{2}Gr_{F}^{4}(\Omega^{\cdot}_{\overline{X}}(\log D) \otimes
\overline{\sV}, \nabla)))}{{\rm {\rm Image}}( H^{2}(\overline{X},
Gr^{W}_{4}Gr_{F}^{4}(\Omega^{\cdot}_{\overline{X}}(\log D) \otimes
\overline{\sV}, \nabla)) \rightarrow H^{3}(\overline{X},
Gr^{W}_{3}Gr_{F}^{4}(\Omega^{\cdot}_{\overline{X}}(\log D) \otimes
\overline{\sV}, \nabla)))}. $$ Since $H^{4}(\overline{X},
Gr^{W}_{2}Gr_{F}^{4}(\Omega^{\cdot}_{\overline{X}}(\log D) \otimes
\overline{\sV}, \nabla)))= H^{2}(\overline{X},
Gr^{W}_{4}Gr_{F}^{4}(\Omega^{\cdot}_{\overline{X}}(\log D) \otimes
\overline{\sV}, \nabla)) =0$, we get
$$Gr_{F}^{4}Gr_{\overline{F}}^{4}Gr^{W[3+2]}_{4+4}(H^{3}(X,\V)) = H^{3}(\overline{X},
Gr^{W}_{3}Gr_{F}^{4}(\Omega^{\cdot}_{\overline{X}}(\log D) \otimes
\overline{\sV}, \nabla)))=H^{1}(D,L^{\otimes 3}_{|D}).$$  Therefore
$h^{4,4}_{3}=h$, where $h$ is the number of connected components of $D$. \\

\noindent 2. $l=2$ \\
As in the case $l=3$ we know by Prop.~\ref{weightcomplex1}
together with Cor.~\ref{vancor} that all the Hodge numbers
$h^{3,4}_{2}$, $h^{5,2}_{2}$, $h^{1,6}_{2}$, $h^{7,0}_{2}$,
$h^{3,3}_{2}$, $h^{1,5}_{2}$, $h^{6,0}_{2}$, $h^{3,2}_{2}$,
$h^{1,4}_{2}$, $h^{5,0}_{2}$ are zero.
So by Hodge symmetry we only need to compute $h^{4,2}_{2}$. \\
Since $Gr_{F}^{4}Gr_{\overline{F}}^{2}Gr^{W[2+2]}_{4+2}(H^{2}(X,\V)) =
$
$$\frac{{\rm Ker}(
H^{2}(\overline{X},
Gr^{W}_{2}Gr_{F}^{4}(\Omega^{\cdot}_{\overline{X}}(\log D) \otimes
\overline{\sV}, \nabla)) \rightarrow H^{3}(\overline{X},
Gr^{W}_{1}Gr_{F}^{4}(\Omega^{\cdot}_{\overline{X}}(\log D) \otimes
\overline{\sV}, \nabla)))}{{\rm {\rm Image}}( H^{1}(\overline{X},
Gr^{W}_{3}Gr_{F}^{4}(\Omega^{\cdot}_{\overline{X}}(\log D) \otimes
\overline{\sV}, \nabla)) \rightarrow H^{2}(\overline{X},
Gr^{W}_{2}Gr_{F}^{4}(\Omega^{\cdot}_{\overline{X}}(\log D) \otimes
\overline{\sV}, \nabla)))}. $$

In addition, $H^{3}(\overline{X},
Gr^{W}_{1}Gr_{F}^{4}(\Omega^{\cdot}_{\overline{X}}(\log D) \otimes
\overline{\sV}, \nabla))= H^{1}(\overline{X},
Gr^{W}_{3}Gr_{F}^{4}(\Omega^{\cdot}_{\overline{X}}(\log D) \otimes
\overline{\sV}, \nabla)) =0,$ hence we get
$$Gr_{F}^{4}Gr_{\overline{F}}^{2}Gr^{W[2+2]}_{4+2}(H^{2}(X,\V)) =
H^{2}(\overline{X},
Gr^{W}_{2}Gr_{F}^{4}(\Omega^{\cdot}_{\overline{X}}(\log D) \otimes
\overline{\sV}, \nabla))= H^{0}(D, \sO_D). $$ Therefore
$h^{4,2}_{2}=h^{2,4}_{2}= h$, where $h$ is the number of connected
components of $D$.
\end{proof}

\medskip {\bf Acknowledgement:}  We thank Jian--Shu Li for discussions in Hongkong and a referee for helpful comments.
This work was supported in DFG Schwerpunkt program ``Global methods in complex geometry''
and Sonderforschungsbereich SFB/TRR 45.

\end{document}